\documentclass[12pt]{amsart}

\usepackage{amssymb}
\usepackage{amsmath,amscd}
\usepackage{color}
\usepackage{epsf,epsfig}
\usepackage{graphicx}
\theoremstyle{plain}
\newtheorem{thm}{Theorem}[section]
\newtheorem{lem}[thm]{Lemma}
\newtheorem{prop}[thm]{Proposition}
\newtheorem{cor}[thm]{Corollary}

\theoremstyle{definition}
\newtheorem{dfn}{Definition}[section]

\newtheorem{exmp}{Example}[section]

\newtheorem{rem}{Remark}[section]

\newcommand{\C}{\mathbb C}

\newcommand{\R}{\mathbb R}
\newcommand{\Z}{\mathbb Z}

\begin{document}

\title [Topological types of 3-dimensional small covers] {Topological types of 3-dimensional small covers}

\author{Zhi L\"u}
\address{Institute of Mathematics, School of Mathematical Sciences, Fudan
 University, Shanghai, 200433, P.R.China.}
\email{zlu@fudan.edu.cn}
\author{Li Yu}
\address{Department of Mathematics and IMS, Nanjing University, Nanjing, 210093, P.R.China}
\email{yuli@nju.edu.cn}

\thanks{Supported by grants from NSFC (No.
 10371020, No. 10671034 and No. 10826040).
 }

\begin{abstract}
In this paper we study the (equivariant) topological types of a
class of 3-dimensional closed manifolds (i.e., 3-dimensional small
covers), each of which admits a locally standard
$(\mathbb{Z}_2)^3$-action such that its orbit space is a simple
convex 3-polytope. We introduce six equivariant operations on
3-dimensional small covers. These six operations are interesting
because of their combinatorial natures. Then we show that each
3-dimensional small cover can be obtained from $\mathbb{R}P^3$ and
$S^1\times\mathbb{R}P^2$ with certain $(\mathbb{Z}_2)^3$-actions
under these six operations. As an application, we classify all
3-dimensional small covers up to $({\Bbb Z}_2)^3$-equivariant
unoriented cobordism.
\end{abstract}

\maketitle

\section{Introduction}\label{sec1}
Small covers are  a class of particularly nicely behaving manifolds
$M^n (n>0)$, introduced by Davis and Januszkiewicz \cite{DJ}, each
of which is an $n$-dimensional closed manifold with a locally
standard $(\mathbb{Z}_2)^n$-action such that its orbit space is a
simple convex $n$-polytope $P^n$. There are strong links of small
covers with combinatorics and polytopes. Davis and Januszkiewicz
showed that small covers have very beautifully algebraic topology.
For example, the equivariant cohomology of a small cover $\pi:
M^n\longrightarrow P^n$ is exactly isomorphic to the Stanley-Reisner
face ring of $P^n$, and the mod 2 Betti numbers $(b_0, b_1,...,b_n)$
of $M^n$ agree with the $h$-vector $(h_0,h_1,...,h_n)$ of $P^n$. In
addition, they also showed that each small cover $\pi:
M^n\longrightarrow P^n$ determines a characteristic function
$\lambda$ (here we call it a $(\mathbb{Z}_2)^n$-coloring)  on $P^n$,
defined by mapping all facets (i.e., $(n-1)$-dimensional faces) of
$P^n$ to nonzero elements of $(\mathbb{Z}_2)^n$ such that $n$ facets
meeting at each vertex are mapped to $n$ linearly independent
elements, and conversely, up to equivariant homeomorphism, $M^n$ can
be reconstructed from the pair $(P^n,\lambda)$.  More specifically,
take a point $x$ in the boundary $\partial P^n$, then there must be
a $l$-dimensional face $F^l$ of $P^n$ such that $x$ is in the
relative interior of $F^l$, where $0\leq l\leq n-1$. Since $P^n$ is
simple (i.e., the number of facets meeting at each vertex is exactly
$n$), there are $n-l$ facets $F_1,...,F_{n-l}$ such that
$F^l=F_1\cap\cdots\cap F_{n-l}$. Let $G_{F^l}$  denote the
rank-$(n-l)$ subgroup of $({\mathbb Z}_2)^n$ determined by
$\lambda(F_1),..., \lambda(F_{n-l})$. Then we can define an
equivalence relation $\sim$ on the product bundle $
P^n\times({\mathbb Z}_2)^n$ as follows:
$$(x,g)\sim (y,h)\Longleftrightarrow \begin{cases}
x=y\text{ and } g=h & \text{ if $x$ is  contained in the interior of $P^n$}\\
x=y\text{ and } gh^{-1}\in G_{F^l} & \text{ if $x$ is contained in
the relative interior of
 }\\
&\text{ some face $F^l\subset \partial P^n$.}
\end{cases}$$
Furthermore, the quotient space $ P^n\times({\mathbb Z}_2)^n/\sim$
denoted by $M(P^n,\lambda)$ recovers $M^n$ up to equivariant
homeomorphism. Geometrically, $M(P^n,\lambda)$ is exactly obtained
by gluing $2^n$ copies of $P^n$ along their boundaries by using
$(\mathbb{Z}_2)^n$-coloring $\lambda$.
 This reconstruction of
small covers provides a way of studying closed manifolds by using
$(\mathbb{Z}_2)^n$-colored polytopes.

\vskip .2cm In \cite{Izmestiev01}, Izmestiev studied a class of
3-dimensional small covers such that each $\lambda$ of
$(\mathbb{Z}_2)^3$-colorings on their orbit spaces is 3-colorable
(i.e.,  the image of $\lambda$ contains only three linearly
independent elements of $(\mathbb{Z}_2)^3$). Such a class of small
covers are ``pullbacks from the linear model" in the terminology of
Davis and Januszkiewicz. Izmestiev obtained a classification result,
saying  that every such small cover can be produced from finitely
many 3-dimensional tori with the canonical $(\mathbb{Z}_2)^3$-action
under  the equivariant connected sum and the equivariant Dehn
surgery.

\vskip .1cm

In this paper, we shall consider all possible 3-dimensional small
covers. Our objective is to determine the (equivariant) topological
types of such a class of 3-dimensional manifolds. Four Color Theorem
guarantees that each simple convex 3-polytope always admits a
$(\mathbb{Z}_2)^3$-coloring. Thus, by the reconstruction of small
covers,  simple convex 3-polytopes with $(\mathbb{Z}_2)^3$-colorings
can recover  all 3-dimensional small covers, so all simple convex
3-polytopes will be involved in studying 3-dimensional small covers.
Throughout this paper, we use the convention that if two simple
convex polytopes $P_1^3$ and $P_2^3$ are combinatorially equivalent,
then $P_1^3$ is identified with $P_2^3$.

\vskip .2cm

Let $\mathcal{P}$ denote the set of all pairs $(P^3, \lambda)$ where
$P^3$ is a simple convex 3-polytope and $\lambda$ is a
$(\mathbb{Z}_2)^3$-coloring on it, and let $\mathcal{M}$ denote the
set of all 3-dimensional small covers. Then, there is a one-to-one
correspondence between $\mathcal{P}$ and $\mathcal{M}$ by mapping
$(P^3, \lambda)$ to $M(P^3,\lambda)$. There is a natural action of
$\text{GL}(3,\mathbb{Z}_2)$ on $\mathcal{P}$, defined by the
correspondence $(P^3, \lambda)\longmapsto (P^3,\sigma\circ\lambda)$
where $\sigma\in\text{GL}(3,\mathbb{Z}_2)$. Obviously, this action
is free, and it also induces an action of
$\text{GL}(3,\mathbb{Z}_2)$ on $\mathcal{M}$ by mapping $M(P^3,
\lambda)$ to $M(P^3, \sigma\circ\lambda)$. Both $M(P^3, \lambda)$
and $M(P^3, \sigma\circ\lambda)$ are $\sigma$-equivariantly
homeomorphic (cf \cite{DJ}), so they are homeomorphic by forgetting
their $(\mathbb{Z}_2)^3$-actions. All elements of each equivalence
class of $\mathcal{P}/\text{GL}(3,\mathbb{Z}_2)$ (resp.
$\mathcal{M}/\text{GL}(3,\mathbb{Z}_2)$) are said to be {\em
$\text{\rm GL}(3,\mathbb{Z}_2)$-equivalent}.

\vskip .1cm

We shall first carry out our work on $\mathcal{P}$.  We shall
introduce six operations $\sharp^v, \sharp^e$, $\sharp^{eve}$,
$\natural$, $\sharp^\triangle$, $\sharp^\copyright$ on
$\mathcal{P}$. Then, under these six operations, up to $\text{\rm
GL}(3,\mathbb{Z}_2)$-equivalence we find five basic pairs
$(\Delta^3,\lambda_0)$, $(P^3(3),\lambda_1)$, $(P^3(3),\lambda_2)$,
$(P^3(3),\lambda_3)$, $(P^3(3),\lambda_4)$ of $\mathcal{P}$
where
$\Delta^3$ is a 3-simplex,  $P^3(3)$ is a 3-sided prism,
and $\lambda_i, i=0,1,...,4$, are shown as in the following figure:
  \[   \input{generators.pstex_t}\centering
   \]
   where $\{e_1,e_2,e_3\}$ is the standard basis of $(\mathbb{Z}_2)^3$.
Then the combinatorial version of our main result is stated as
follows.

\begin{thm}\label{T1}
Each pair $(P^3, \lambda)$ in $\mathcal{P}$ is an expression of
$(\Delta^3,\sigma\circ\lambda_0)$, $(P^3(3),\sigma\circ\lambda_1)$,
$(P^3(3),\sigma\circ\lambda_2)$, $(P^3(3),\sigma\circ\lambda_3)$,
$(P^3(3),\sigma\circ\lambda_4)$,
$\sigma\in \text{\rm
GL}(3,\mathbb{Z}_2)$, under six operations $\sharp^v$, $\sharp^e$,
$\sharp^{eve}$,$\natural$, $\sharp^\triangle$, $\sharp^\copyright$.
\end{thm}

By the reconstruction of small covers,  six operations $\sharp^v,
\sharp^e, \sharp^{eve}$, $\natural$, $\sharp^\triangle$,
$\sharp^\copyright$ on $\mathcal{P}$ naturally correspond to six
equivariant operations on $\mathcal{M}$, denoted by
$\widetilde{\sharp^v}, \widetilde{\sharp^e},
\widetilde{\sharp^{eve}}, \widetilde{\natural},
\widetilde{\sharp^\triangle}, \widetilde{\sharp^\copyright}$,
respectively. These six operations can be understood very well
because of their combinatorial natures. We shall see that
$\widetilde{\sharp^v}$ is the equivariant connected sum, and
$\widetilde{\natural}$ is the equivariant Dehn surgery, and other
four operations $\widetilde{\sharp^e}, \widetilde{\sharp^{eve}},
\widetilde{\sharp^\triangle}, \widetilde{\sharp^\copyright}$ can be
understood as the generalized equivariant connected sums. By
Theorem~\ref{T1},  $M(\Delta^3,\sigma\circ\lambda_0)$ and
$M(P^3(3),\sigma\circ\lambda_i) (i=1,...,4)$,
 $\sigma\in \text{\rm
GL}(3,\mathbb{Z}_2)$, give all elementary generators of the
algebraic system $\langle\mathcal{M}; \widetilde{\sharp^v},
\widetilde{\sharp^e}, \widetilde{\sharp^{eve}},
\widetilde{\natural}, \widetilde{\sharp^\triangle},
\widetilde{\sharp^\copyright}\rangle $.  On the other hand, we shall
show that $M(\Delta^3,\lambda_0)$ is equivariantly homeomorphic to
the $\mathbb{R}P^3$ with canonical linear $(\mathbb{Z}_2)^3$-action,
and $M(P^3(3),\lambda_i), i=1,...,4$, are equivariantly homeomorphic
to the $S^1\times\mathbb{R}P^2$ with four different
$(\mathbb{Z}_2)^3$-actions respectively.
 Then the topological version of our
main result is stated as follows.

\begin{thm}\label{T2}
Each 3-dimensional small cover can be obtained from $\mathbb{R}P^3$
and $S^1\times \mathbb{R}P^2$ with certain
$(\mathbb{Z}_2)^3$-actions by using six operations
$\widetilde{\sharp^v}, \widetilde{\sharp^e},
\widetilde{\sharp^{eve}}, \widetilde{\natural},
\widetilde{\sharp^\triangle}, \widetilde{\sharp^\copyright}$.
\end{thm}

\begin{rem}
Theorem~\ref{T2} tells us how to obtain a 3-dimensional small cover
from only two  known small covers $\mathbb{R}P^3$ and $S^1\times
\mathbb{R}P^2$ with certain actions by using cut and paste
strategies in the sense of six equivariant operations. This is an
equivariant analogue of a well-known result (\cite{Lickorish62},
\cite{Lickorish63}, see also \cite{kirby} and \cite{Roberts}) as
follows: ``Each closed 3-manifold can be obtained from a 3-sphere
$S^3$ or a $S^3$ with one non-orientable bundle by using a finite
number of Dehn surgeries''.
\end{rem}

As an application, we study the $(\mathbb{Z}_2)^3$-equivariant
unoriented cobordism classification of all 3-dimensional small
covers. Let $\widehat{\mathcal{M}}$ denote the set of
$(\mathbb{Z}_2)^3$-equivariant unoriented cobordism classes of all
3-dimensional small covers. Then $\widehat{\mathcal{M}}$ forms an
abelian group under disjoint union,
 so it is also a vector space over $\Z_2$.

\begin{thm} \label{T3}
 $\widehat{\mathcal{M}}$ is generated by classes of $\mathbb{R}P^3$
and $S^1\times \mathbb{R}P^2$ with certain
$(\mathbb{Z}_2)^3$-actions.
\end{thm}

\begin{rem}
It should be pointed out that Theorem~\ref{T3} is a direct corollary
of main theorems in \cite{Lu-cobordism}, but here we shall give it a
different proof. Actually, the first author of this paper in
\cite{Lu-cobordism} dealt with general closed 3-manifolds with
effective $(\Z_2)^3$-actions. Let $\mathfrak{M}_3$ be the
$\Z_2$-vector space consisting of $(\mathbb{Z}_2)^3$-equivariant
unoriented cobordism classes of all closed 3-manifolds with
effective $(\Z_2)^3$-actions. Then it was shown in
\cite{Lu-cobordism} that $\mathfrak{M}_3$ can be generated by
classes of $\mathbb{R}P^3$ and $S^1\times \mathbb{R}P^2$ with
certain $(\mathbb{Z}_2)^3$-actions, and each class of
$\mathfrak{M}_3$ contains a small cover as its representative. In
particular, it was also shown in \cite{Lu-cobordism} that
$\mathfrak{M}_3$ has dimension 13. Thus, $\widehat{\mathcal{M}}$ has
dimension 13, too.
\end{rem}

This paper is organized as follows. In Section~\ref{sec2} we
establish the six operations on $\mathcal{P}$, and then we prove
Theorem~\ref{T1} in Section~\ref{sec3}. In Section~\ref{sec4} we
study elementary colored 3-polytopes
and determine their equivariant topological types. Moreover,
Theorem~\ref{T2} is settled. In Section~\ref{sec5} we discuss how
the corresponding six equivariant operations work on $\mathcal{M}$.
As an application, we consider the $(\mathbb{Z}_2)^3$-equivariant
unoriented cobordism classification of all 3-dimensional small
covers and prove Theorem~\ref{T3} in Section~\ref{sec6}.

\section{Operations on $\mathcal{P}$} \label{sec2}

The task of this section is to define  six operations on
$\mathcal{P}$. Throughout the remaining part of this paper, each
nonzero element of $(\mathbb{Z}_2)^3$ is called a {\em color}, so
$(\mathbb{Z}_2)^3$ contains seven colors.

\vskip .1cm

First, let us look at all simple uncolored 3-polytopes. It is
well-known that any simple convex 3-polytope can be obtained from a
3-simplex by using three types of {\em excision} methods illustrated
in the following figure:
 cutting out (i) a vertex; (ii) an edge; (iii) two edges with a common
 vertex. See Gr\"unbaum's book \cite[p.270]{Baum}.
  \[   \begin{picture}(0,0)%
\includegraphics{cutting1.pstex}%
\end{picture}%
\setlength{\unitlength}{1342sp}%
\begingroup\makeatletter\ifx\SetFigFont\undefined%
\gdef\SetFigFont#1#2#3#4#5{%
  \reset@font\fontsize{#1}{#2pt}%
  \fontfamily{#3}\fontseries{#4}\fontshape{#5}%
  \selectfont}%
\fi\endgroup%
\begin{picture}(7566,3123)(7768,-3451)
\put(10051,-3361){\makebox(0,0)[lb]{\smash{{\SetFigFont{7}{8.4}{\rmdefault}{\mddefault}{\updefault}Cutting out a vertex }}}}
\end{picture}%
\centering
   \]
    \[   \begin{picture}(0,0)%
\includegraphics{cutting2.pstex}%
\end{picture}%
\setlength{\unitlength}{1342sp}%
\begingroup\makeatletter\ifx\SetFigFont\undefined%
\gdef\SetFigFont#1#2#3#4#5{%
  \reset@font\fontsize{#1}{#2pt}%
  \fontfamily{#3}\fontseries{#4}\fontshape{#5}%
  \selectfont}%
\fi\endgroup%
\begin{picture}(7273,3348)(8968,-3376)
\put(11251,-3286){\makebox(0,0)[lb]{\smash{{\SetFigFont{7}{8.4}{\rmdefault}{\mddefault}{\updefault}Cutting out an edge}}}}
\end{picture}%
\centering
   \]
    \[   \begin{picture}(0,0)%
\includegraphics{cutting3.pstex}%
\end{picture}%
\setlength{\unitlength}{1342sp}%
\begingroup\makeatletter\ifx\SetFigFont\undefined%
\gdef\SetFigFont#1#2#3#4#5{%
  \reset@font\fontsize{#1}{#2pt}%
  \fontfamily{#3}\fontseries{#4}\fontshape{#5}%
  \selectfont}%
\fi\endgroup%
\begin{picture}(7866,3798)(8968,-4126)
\put(9526,-4036){\makebox(0,0)[lb]{\smash{{\SetFigFont{7}{8.4}{\rmdefault}{\mddefault}{\updefault}Cutting out two edges with a common vertex}}}}
\end{picture}%
\centering
   \]
   Since we shall carry out our study on colored polytopes and small covers,  although these three types of excisions are very simple,
  they cannot directly work on colored polytopes and small covers
  because they will destroy the closedness
  of small covers.
 However, for our purpose we can  interpret them as
    the \textquotedblleft connected sum\textquotedblright operations with
    some standard simple $3$-polytopes as follows.

\subsection{Three operations $\sharp^v$, $\sharp^e$
      and $\sharp^{eve}$}
   \begin{enumerate}
    \item[(I)] The operation $\sharp^v$ with a 3-simplex $\Delta^3$
     \[   \input{Cut1.pstex_t}\centering
   \]

     \item[(II)] The operation $\sharp^e$ with a 3-sided prism $P^3(3)$
      \[   \input{Cut2.pstex_t}\centering
   \]

      \item[(III)] The operation $\sharp^{eve}$ with a truncated prism
      $P^3_-(3)$
     \[   \input{Cut3.pstex_t}\centering
   \]
    \end{enumerate}
      Obviously,  each of three operations is invertible as long as we don't perform  the corresponding inverse
      operations
      of $\sharp^v, \sharp^e,  \sharp^{eve}$ on $\Delta^3$,
      $P^3(3)$, $P^3_-(3)$, respectively.
      Also,
     We always can do the operation
      $\sharp^v$ between any two simple 3-polytopes. Since a 3-sided
      prism and a truncated prism can be obtained from a 3-simplex
      by using the operation
      $\sharp^v$, we have

      \begin{prop}\label{non-coloring}
      Each simple 3-polytope can be obtained
      from a 3-simplex under three operations $\sharp^v$, $\sharp^e$
      and $\sharp^{eve}$.
\end{prop}

\begin{dfn}
Let $P^3$ be a simple  3-polytope, and let $F$ be a facet of $P^3$.
Then $F$ is a $\ell$-polygon where $\ell\geq 3$. If $\ell\leq 5$,
then $F$ is called a {\em small facet}; otherwise, it is called a
{\em big facet}.
\end{dfn}

Also, for two edges
      with a common vertex in a simple 3-polytope  $P^3$, denoted by $V_{eve}$, there are at least four neighboring facets
      around $V_{eve}$. But it is easy to see that there are
      exactly five neighboring facets around $V_{eve}$ if $V_{eve}$ is not in a triangular facet.
      Since $V_{eve}$ is always associated with the operation $\sharp^{eve}$, throughout the rest of the paper,
      we use the {\em convention} that  $V_{eve}$ must be chosen in an $m$-polygonal facet with $m\geq
      5$,  so there are five neighboring
facets around $V_{eve}$.

\vskip .2cm

Suppose that $P^3$ is a simple 3-polytope but it is not a 3-simplex.
Then we know by Proposition~\ref{non-coloring} that $P^3$ comes from
applying one of the  three types of cutting operations on some
simple 3-polytope $P'^3$ such that the number of facets of $P'^3$ is
one less than that of $P^3$. In other words, there is a small facet
$F$ of $P^3$ such that  $P'^3$ is obtained by compressing $F$ into a
point, or an edge or a $V_{eve}$ in $P^3$. In this case, we say that
$P^3$ is {\em compressible at $F$}, and $P'^3$ is the {\em
compression of $P^3$ at $F$}.

\begin{cor}
Suppose that $P^3$ is a simple 3-polytope other than a 3-simplex.
Then $P^3$ is always  compressible at some small facet.
\end{cor}

\vskip .2cm

      Now let us carry out our work on $\mathcal{P}$. We wish to know how the three operations $\sharp^v$, $\sharp^e$
      and $\sharp^{eve}$ work
       on $\mathcal{P}$.  To make
our language more concise,  first let us give some notions.

      \begin{dfn} [Local colorings] Given a pair $(P^3,\lambda)$ in $\mathcal{P}$. Let
      $v$ be a vertex (or a 0-face) of $P^3$. The
      colors of three facets meeting at $v$ are said to be a
      {\em coloring of $v$}. Let $e$ be an edge (or a 1-face) of $P^3$. Then there must be four neighboring facets around $e$ since
      $P^3$ is simple, and the colors of these four facets are
      said to be a {\em coloring of $e$}. Similarly, for a $V_{eve}$ in $P^3$,  there are  five neighboring facets
      around $V_{eve}$, and then the colors of those facets are said to be a
      {\em coloring of $V_{eve}$}.
      \end{dfn}

     \begin{rem}\label{local coloring}
By the definition of $(\mathbb{Z}_2)^3$-colorings, the colors of
neighboring  facets around a vertex  (resp.  an edge  and  a
$V_{eve}$) always can span the whole
      $(\mathbb{Z}_2)^3$. It is easy to see that up to $\text{\rm
      GL}(3,\mathbb{Z}_2)$-equivalence, a vertex admits a unique
      coloring, an edge admits four different kinds of  colorings, and a
      $V_{eve}$ admits 16 different kinds of colorings. We list them
      as follows:
       \begin{enumerate}
\item Colorings of a vertex  and an edge
\[   \input{f1-1.pstex_t}\centering
   \]
\item Colorings of a $V_{eve}$
 \[   \input{f1.pstex_t}\centering
   \] 
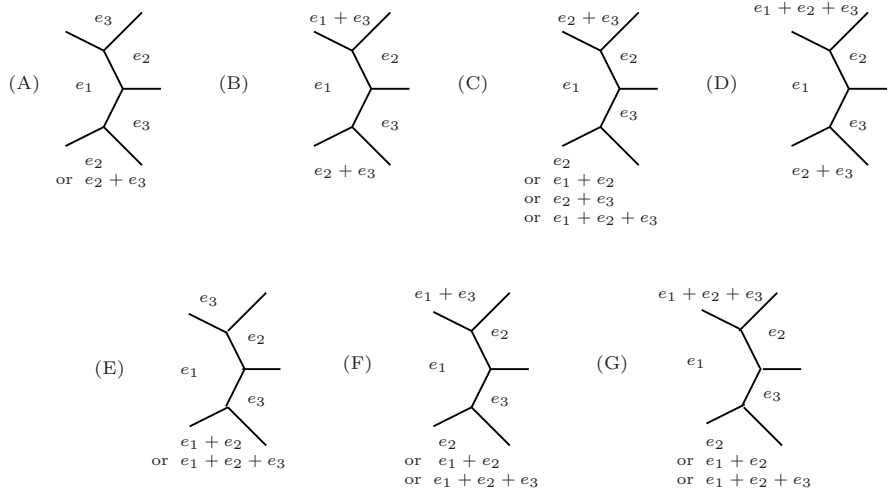
 \[   \input{ff1.pstex_t}\centering
   \]
        \end{enumerate}
\end{rem}

\begin{dfn}
Given a pair $(P^3,\lambda)$ in $\mathcal{P}$, and let $F$ be a
facet of $P^3$.   $F$ is said to be {\em 2-independent} if the
colors of the neighboring facets around $F$ span a 2-dimensional
subspace of $(\mathbb{Z}_2)^3$. Similarly, $F$ is said to be {\em
3-independent} if the colors of the neighboring facets around $F$
span the whole $(\mathbb{Z}_2)^3$.
\end{dfn}

With the above understood,  let us look at how the three operations
$\sharp^v$, $\sharp^e$
      and $\sharp^{eve}$ work
       on $\mathcal{P}$.

\begin{prop} \label{operation-1}
Up to $\text{\rm GL}(3,\mathbb{Z}_2)$-equivalence, the first two
operations  $\sharp^v$ and $\sharp^e$ can operate on any vertex and
any edge in a colored simple 3-polytope, respectively,  and the
third operation $\sharp^{eve}$ can operate on $V_{eve}$ in a colored
simple 3-polytope as long as the coloring of $V_{eve}$ does not
match any one of eight kinds of colorings shown in the figures
(E)-(G) of Remark~\ref{local coloring}(2).
\end{prop}
\begin{proof}
Let $(P^3, \lambda)$ be a pair  in $\mathcal{P}$. Choose a vertex
$v$ of $P^3$, since $v$ admits a unique coloring up to $\text{\rm
GL}(3,\mathbb{Z}_2)$-equivalence,  there is a pair $(\Delta^3,
\lambda')$ such that some vertex in $\Delta^3$ has the same coloring
as $v$, so that we can do the operation $\sharp^v$ between $(P^3,
\lambda)$ and $(\Delta^3, \lambda')$. Choose an edge $e$ of $P^3$,
then we know from Remark~\ref{local coloring}(1) that there are four
kinds of colorings of $e$ up to $\text{\rm
GL}(3,\mathbb{Z}_2)$-equivalence, which agree with those colorings
of an edge $e'$ of $P^3(3)$, as shown in Section 1, where $e'$ is
not an edge of any triangle facet of $P^3(3)$. Thus, we
 can perform the operation $\sharp^e$ on $(P^3, \lambda)$ with some pair
$(P^3(3), \lambda'')$. Choose a $V_{eve}$ (i.e., two edges with a
common vertex) in some facet $F$ of $P^3$ (note that $F$ is an
$m$-polygon with $m\geq 5$ by our convention as before). We know
from Remark~\ref{local coloring}(2) that there are 16 kinds of
colorings of $V_{eve}$ up to $\text{\rm
GL}(3,\mathbb{Z}_2)$-equivalence. However, it is easy to see that
the eight kinds of colorings shown in the figures (E)-(G) cannot be
used as colorings of the neighboring facets around a pentagon in a
simple 3-polytope by the definition of $(\mathbb{Z}_2)^3$-colorings.
 This means that if $V_{eve}$ has such a coloring,  we can not perform the operation $\sharp^{eve}$ on  $(P^3, \lambda)$
 at $V_{eve}$. On the other hand,
consider a $V_{eve}'$ in a truncated prism as shown in the following
figure:
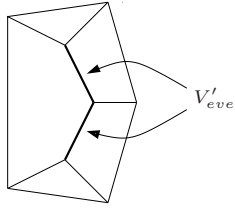
\[   \begin{picture}(0,0)%
\includegraphics{f2.pstex}%
\end{picture}%
\setlength{\unitlength}{1184sp}%
\begingroup\makeatletter\ifx\SetFigFont\undefined%
\gdef\SetFigFont#1#2#3#4#5{%
  \reset@font\fontsize{#1}{#2pt}%
  \fontfamily{#3}\fontseries{#4}\fontshape{#5}%
  \selectfont}%
\fi\endgroup%
\begin{picture}(5660,4244)(4179,-4883)
\put(8101,-2761){\makebox(0,0)[lb]{\smash{{\SetFigFont{7}{8.4}{\rmdefault}{\mddefault}{\updefault}$V_{eve}'$}}}}
\end{picture}%
\centering
   \]
Obviously, $V_{eve}'$ admits those  eight kinds of colorings shown
in the figures (A)-(D) of Remark~\ref{local coloring}(2), but it
admits none of eight kinds of colorings shown in the figures (E)-(G)
of Remark~\ref{local coloring}(2).
   Therefore, $(P^3, \lambda)$ can do the operation
$\sharp^{eve}$ with a $P^3_-(3)$ with some
$(\mathbb{Z}_2)^3$-coloring if $V_{eve}$ admits a coloring which
matches one of the eight kinds of colorings shown in the figures
(A)-(D) of Remark~\ref{local coloring}(2).
\end{proof}

\begin{rem}
It should be pointed out that $\sharp^v$ can always operate between
any two pairs $(P_1^3, \lambda_1)$ and $(P^3_2, \lambda_2)$ in
$\mathcal{P}$ up to $\text{\rm GL}(3,\mathbb{Z}_2)$-equivalence. In
fact, choose two vertices $v_1$ and $v_2$ in $P_1^3$ and $P^3_2$
respectively, then $v_1$ and $v_2$ have the same coloring up to
$\text{\rm GL}(3,\mathbb{Z}_2)$-equivalence. Thus, by applying an
automorphism $\sigma\in \text{\rm GL}(3,\mathbb{Z}_2)$ to $(P_1^3,
\lambda_1)$,  we can change the coloring of $v_1$ into that of
$v_2$, so that we can do the operation $\sharp^v$ between $(P_1^3,
\sigma\circ\lambda_1)$ and $(P^3_2, \lambda_2)$. We shall see that
$\sharp^v$ exactly agrees with the equivariant connected sum of
3-dimensional small covers.
\end{rem}

Similarly to the uncolored case, clearly  we still cannot perform
the corresponding inverse   operations
      of  $\sharp^v, \sharp^e, \sharp^{eve}$ on colored $\Delta^3$,
      $P^3(3)$,$P^3_-(3)$, respectively. However,  by Remark~\ref{local coloring}, it is easy to see that for any
      small facet $F$ of a pair $(P^3, \lambda)$, if $F$ is 2-independent, then  $(P^3, \lambda)$ cannot be compressed
at $F$.

\vskip .2cm

By Proposition~\ref{non-coloring}, a natural question is whether
      each pair $(P^3,
\lambda)$ of $\mathcal{P}$ can be produced only from a 3-simplex
with $(\mathbb{Z}_2)^3$-colorings in such three operations. However,
generally the answer is {\em no}. For example, none  of the four
colorings on $P^3(3)$ as shown in Section 1 can be obtained from a
3-simplex with $(\mathbb{Z}_2)^3$-colorings under three operations
$\sharp^v$, $\sharp^e$ and $\sharp^{eve}$. This is because  each
triangular facet in $P^3(3)$ with any one of those four colorings is
2-independent and it cannot be compressed into a point. More
generally, we can further ask the following question:
 \begin{enumerate}
 \item[(Q1):] {\em Can any pair $(P^3,
\lambda)$  be produced by a $3$-simplex,
    a prism and a truncated prism with $(\mathbb{Z}_2)^3$-colorings under operations $\sharp^v$, $\sharp^e$
      and $\sharp^{eve}$?}
 \end{enumerate}
 Unfortunately, the answer is still {\em no}.  Actually, generally
 it is possible that all the small facets are 2-independent, so we can not do the compression of  $(P^3, \lambda)$
 at any of its small facets at all.  This can be seen from the following
      example.
\begin{exmp}
      Consider two copies of a square with four
      neighboring 6-polygons, we can glue them into a simple
      3-polytope admitting a 3-colorable coloring, as shown in the
      following figure:
      \[   \input{f3.pstex_t}\centering
   \]
However, this 3-colorable example  can not be compressed at any
facet under operations $\sharp^v$, $\sharp^e$
      and $\sharp^{eve}$ since each coloring on a 3-simplex
      (resp. a 3-sided prism, and a truncated prism) is not 3-colorable.
\end{exmp}

\begin{rem} Generally, when a pair $(P^3, \lambda)$ of $\mathcal{P}$ is
3-colorable, a theorem of Izmestiev in \cite{Izmestiev01} claims
that
    $(P^3, \lambda)$ can be obtained
    from a finite set of 3-colorable cubes by using the equivariant connected sum (i.e., the operation $\sharp^v$)  and
    the equivariant Dehn surgery. The reason why his work was carried out very well is
    because
    the coloring of a $3$-colorable polytope
    is unique up to $\text{\rm GL}(3,\mathbb{Z}_2)$-equivalence, while generally speaking, the set
    of all colorings given by more than three colors is quite complicated.
\end{rem}

\subsection{Operations $\natural$ and $\sharp^\triangle$ on
$\mathcal{P}$} According to the work of Izmestiev
(\cite{Izmestiev01}), we might need the fourth operation $\natural$
on $\mathcal{P}$.   This operation originally comes from the Dehn
surgery on 3-manifolds rather than combinatorics. Based upon the
topological meaning of Dehn surgery, Izmestiev gave it a
combinatorial description by deleting a quarter of a cylinder with a
subsequent gluing of a half-cylinder. Although  this combinatorial
description of the operation $\natural$  works on $\mathcal{P}$ very
well, it doesn't meet the style of this paper, that is, it does not
accord with the descriptions of other operations on $\mathcal{P}$ in
this paper. For this, we give another combinatorial description of
this operation $\natural$, which is shown as follows:
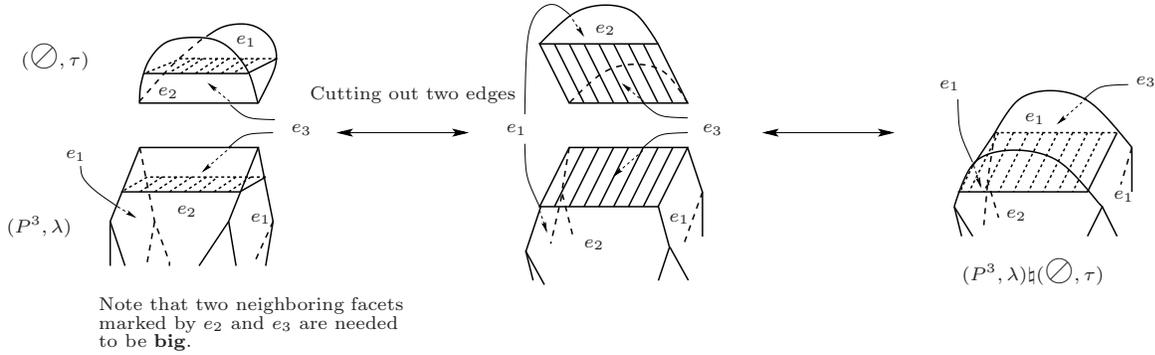
\[   \input{Dehn-sphere.pstex_t}\centering
   \]
where $\text{\Large $\oslash$}$ is a quarter of a 3-ball, whose
boundary consists of three 2-polygons, three edges and two vertices.
Clearly {\Large $\oslash$} is not a simple 3-polytope, but it
   still admits a $(\mathbb{Z}_2)^3$-coloring. Note that {\Large $\oslash$} is actually a nice manifold with
   corners (\cite{D}).
     Obviously, the operation  $\natural$ is invertible. However,  generally it may
not be closed in $\mathcal{P}$ because doing the operation
$\natural$ on a colored 3-polytope $(P^3, \lambda)$ might destroy
the 3-connectedness of the 1-skeleton of  $P^3$. In the 3-colorable
case, Izmestiev showed  that if $\natural$  makes the 1-skeleton of
the polytope  not
           $3$-connected, then one can find a connected sum
         somewhere else in the original polytope. In the general
         case, the argument of Izmestiev can be carried out to get a
         generalized result. The following is the combinatorial
lemma proved by Izmestiev in \cite{Izmestiev01} which will be used
later in this paper.
\begin{lem} [\cite{Izmestiev01}]
If the 1-skeleton of a 3-polytope $P$ is disconnected after
    cutting out three non-adjacent edges, then $P$ can be written as $P=P_1
    \sharp^v
    P_2$, where $P_1, P_2$ are 3-polytopes. In addition, when $P$ is
    simple, so are $P_1$ and $P_2$.
\end{lem}

Next, given a pair $(P^3, \lambda)$ in $\mathcal{P}$, suppose that
we can do an equivariant Dehn surgery on $(P^3, \lambda)$, but this
operation destroys the 3-connectedness of the 1-skeleton $\Gamma$ of
$P^3$. If $\lambda$ is 3-colorable, Izmestiev gave a canonical
method of finding three non-adjacent edges $x_1,x_2,x_3$ of $P^3$
such that $\Gamma\backslash\{x_1,x_2,x_3\}$ is disconnected (see
\cite{Izmestiev01} for the argument in detail). Then there are two
3-colorable pairs $(P_1^3, \lambda_1)$ and $(P_2^3, \lambda_2)$ such
that $(P^3, \lambda)=(P_1^3, \lambda_1)\sharp^v(P_2^3, \lambda_2)$,
as shown in the following figure:
 \[   \input{f4.pstex_t}\centering
   \]

In the general case, we can still use the Izmestiev's method to find
the required three non-adjacent edges $x_1,x_2,x_3$ such that
$\Gamma\backslash\{x_1,x_2,x_3\}$ is disconnected, but  there are
{\em two possible colorings} up to $\text{\rm
GL}(3,\mathbb{Z}_2)$-equivalence for three facets determined by
$x_1,x_2,x_3$, as shown in the following figure:
 \[   \input{f5.pstex_t}\centering
   \]
Obviously, the case (I) is the same as the 3-colorable case above,
so there are two  pairs $(P_1^3, \lambda_1)$ and $(P_2^3,
\lambda_2)$ such that $(P^3, \lambda)=(P_1^3,
\lambda_1)\sharp^v(P_2^3, \lambda_2)$. If the case (II) happens,
then there still are two pairs $(P_1^3, \lambda_1)$ and $(P_2^3,
\lambda_2)$, but we need to introduce a new operation
$\sharp^\triangle$, so that $(P^3, \lambda)$ is equal to the sum of
$(P_1^3, \lambda_1)$ and $(P_2^3, \lambda_2)$ under this new
operation  $\sharp^\triangle$.

\vskip .2cm

The {\em operation $\sharp^\triangle$} is defined as follows: first
we cut out a triangular facet of $(P_i^3, \lambda_i), i=1,2$,
respectively, and then we glue them together along their triangular
sections, as shown in the following figure:
 \[   \input{f6.pstex_t}\centering
   \]
Notice that  the operation $\sharp^\triangle$ is invertible. It
should be pointed out that the operation $\sharp^\triangle$ can also
work in the case (I).

\vskip .2cm

Combining the above argument, we have

\begin{prop} \label{composition}
Let $(P^3, \lambda)$ be a pair in $\mathcal{P}$. Suppose that the
3-connectedness of 1-skeleton of $P^3$ is destroyed after doing an
equivariant Dehn surgery $\natural$ on $(P^3, \lambda)$. Then there
are two pairs $(P_1^3, \lambda_1)$ and $(P_2^3, \lambda_2)$  in
$\mathcal{P}$ such that either $(P^3, \lambda)=(P_1^3,
\lambda_1)\sharp^v(P_2^3, \lambda_2)$ or $(P^3, \lambda)=(P_1^3,
\lambda_1)\sharp^\triangle(P_2^3, \lambda_2)$.
\end{prop}

Next given a pair $(P^3, \lambda)$ in $\mathcal{P}$ and let $F$ be a
small facet. We wish to know
\begin{enumerate}
 \item[(Q2):] {\em Can $(P^3, \lambda)$ always  be compressed at $F$ if $F$ is
3-independent?}
 \end{enumerate}
 To answer this question (Q2), we need to introduce the
following operation.

\subsection{Operation $\sharp^\copyright$---Coloring change on $\mathcal{P}$} Now let us introduce the sixth operation $\sharp^\copyright$ on
$\mathcal{P}$. Given a pair $(P^3, \lambda)$ in $\mathcal{P}$, we
cannot avoid the occurrence of 2-independent facets in $(P^3,
\lambda)$ in general, but for our propose we can change their
colorings. Let $F$ be a 2-independent $l$-polygonal facet of $(P^3,
\lambda)$. Then we can construct a $l$-sided prism $Q=F\times
[0,1]$, which naturally admits a coloring $\tau$ such that the
coloring of the neighboring facets around the top facet (or bottom
facet) is the same as that of $F$ in $(P^3, \lambda)$. Since $F$ is
2-independent, we can give two different colorings on the top facet
and the bottom facet of $Q$, such that
 the bottom facet of $Q$ has the same coloring as $F$.
Then we can define an operation between $(P^3, \lambda)$ and $(Q,
\tau)$ as follows: cutting out the $F$ of $P^3$ and the bottom facet
of $Q$, and then gluing them together along sections, as shown in
the following figure:
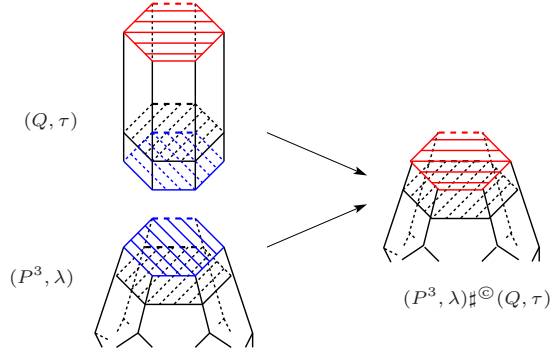
\[   \input{f7.pstex_t}\centering
   \]
This operation exactly changes the coloring of $F$, so we also call
it the {\em coloring change}, denoted by $\sharp^\copyright$.
Clearly, the operation $\sharp^\copyright$ is invertible.

\vskip .2cm

We shall mainly consider the coloring changes of 2-independent small
facets, all possible cases (in the sense of $\text{\rm
GL}(3,\mathbb{Z}_2)$-equivalence) of which are listed  as follows:
 \begin{enumerate}
\item[(a)] triangular  case
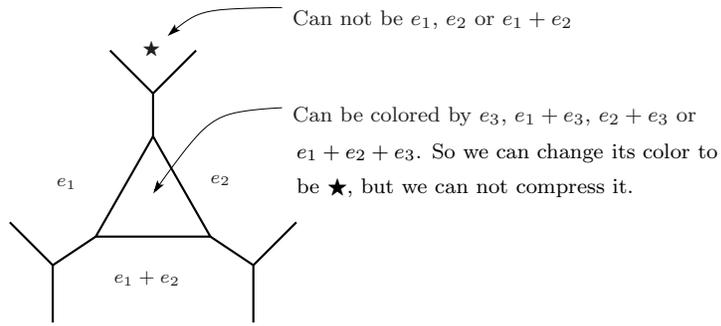
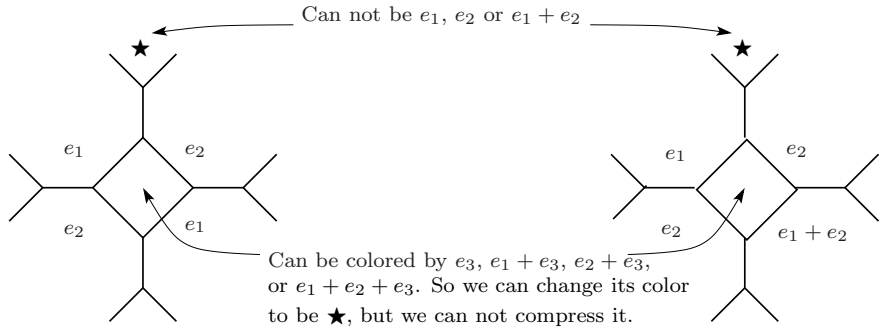
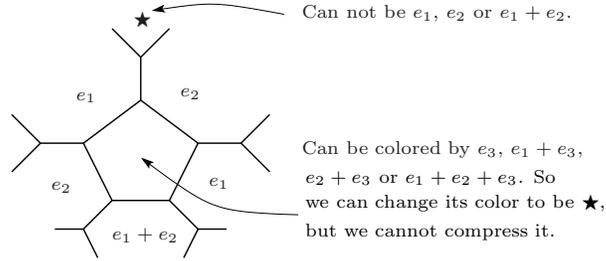
\[   \input{3-gon-color.pstex_t}\centering
   \]
       \item[(b)] rectangular case
       \[   \input{4-gon-color.pstex_t}\centering
   \]
      \item[(c)]pentagonal  case
       \[   \input{5-gon-color.pstex_t}\centering
   \]
     \end{enumerate}

\begin{rem} \label{section1}
As seen as above,   when we do those six operations on
$\mathcal{P}$, we need to cut out vertices, edges, $V_{eve}$'s,
2-independent triangular facets, 2-independent square facets, and
2-independent pentagonal facets, so that we can produce  different
kinds of sections on polytopes. By $S_v, S_e, S_{V_{eve}}$,
$S_\triangle, S_\square$, and $S_\maltese$ we denote those sections
obtained by cutting out a vertex $v$, an edge $e$ and a $V_{eve}$, a
2-independent triangular facet, a 2-independent square facet and a
2-independent pentagonal facet respectively. Also, the colorings of
neighboring facets around $S_v, S_e, S_{V_{eve}}$, $S_\triangle,
S_\square$, and $S_\maltese$ are said to be the {\em colorings} of
$S_v, S_e, S_{V_{eve}}$, $S_\triangle, S_\square$, and $S_\maltese$
respectively. Obviously, these sections have the properties:
\begin{enumerate}
\item  The colorings of  $S_v, S_e, S_{V_{eve}}$ are all 3-independent. Up to $\text{\rm GL}(3,\mathbb{Z}_2)$-equivalence, $S_v$
admits a unique coloring,  $S_e$ admits four different colorings,
and $S_{V_{eve}}$ admits eight different colorings. The colorings of
$S_v$ and $S_e$ agree with the colorings of a vertex and an edge
respectively,  see Remark~\ref{local coloring}(1). The colorings of
$S_{V_{eve}}$ agree with the colorings shown in the figures (A)-(D)
of Remark~\ref{local coloring}(2).

\item The colorings of  $S_\triangle, S_\square,
S_\maltese$ are all 2-independent. Up to $\text{\rm
GL}(3,\mathbb{Z}_2)$-equivalence, $S_\triangle$ and $S_\maltese$
admit a unique coloring, but $S_\square$ admits two different
colorings. These colorings agree with the colorings of around
2-independent small facets, as shown before Remark~\ref{section1}.
\end{enumerate}
Notice that $S_v$ and $S_\triangle$ are triangular sections, $S_e$
and $S_\square$ are rectangular sections, and $S_{V_{eve}}$ and
$S_\maltese$ are pentagonal sections.
\end{rem}

Finally, let us discuss the question (Q2).

\begin{prop}\label{compress}
Let $(P^3, \lambda)$ be a pair in $\mathcal{P}$ and let $F$ be a
small facet of $P^3$. Then $(P^3, \lambda)$ is compressible at $F$
 if and only if $F$ is
3-independent.
\end{prop}
\begin{proof} Since we cannot perform the
corresponding inverse   operations
      of  $\sharp^v, \sharp^e, \sharp^{eve}$ on colored $\Delta^3$,
      $P^3(3)$,$P^3_-(3)$, respectively, we may assume that  when
$F$ is a triangular (resp. rectangular, or pentagonal) facet, $P^3$
is not a 3-simplex (resp. a $P^3(3)$, or a $P^3_-(3)$). Obviously,
if $(P^3, \lambda)$ is compressible at $F$ then $F$ is
3-independent. Conversely, our argument  proceeds as follows.

\vskip .2cm (1)  Suppose that $F$ is a 3-independent triangular
facet. Then it is easy to see that $(P^3, \lambda)$ is compressible
at $F$.

\vskip .2cm (2) Suppose that $F$ is  a 3-independent rectangular
facet. Then up to $\text{GL}(3, {\Bbb Z}_2)$-equivalence, we may
list all possible colorings of $F$ and its four neighboring facets
 as follows:
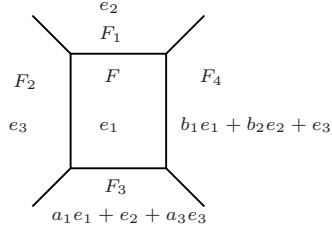
\[   \input{rect1.pstex_t}\centering
   \]
where $a_i, b_j\in {\Bbb Z}_2$ with $b_2a_3=0$,  and at least one of
$a_1$ and $b_1$ is nonzero. Obviously, if none  of the four
neighboring facets around $F$ is  triangular, then one can always
compress $F$ into an edge along $F_3$ (if $a_1\not=0$) or $F_4$ (if
$b_1\not=0$), as shown in the following figure.
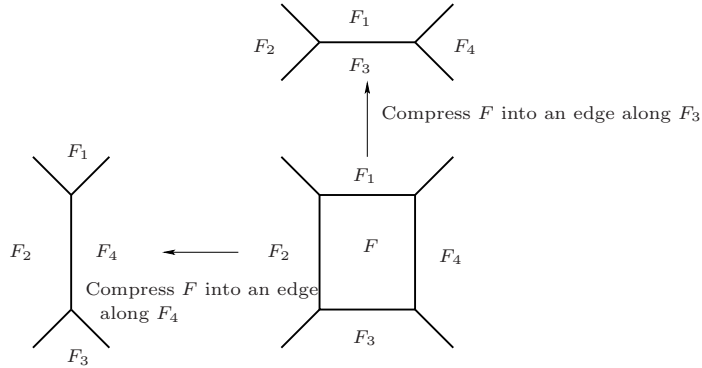
\[   \input{rect2.pstex_t}\centering
   \]
   If there are triangular neighboring facets
around $F$, then by Steinitz's Theorem the number of such triangular
neighboring facets is at most 2. Since $P^3$ is not a $P^3(3)$ by
our assumption, the number must be 1. With no loss assume that $F_1$
is a triangular facet. If $F_1$ is 3-independent, then one first
compress $F_1$ into a point, so that $F$ becomes a 3-independent
triangular facet. Furthermore, one can compress $F$ into a point. If
$F_1$ is 2-independent and $a_1=1$, then one can compress $F$ into
an edge along $F_3$; if $F_1$ is 2-independent but $a_1=0$, then one
can first change the coloring of $F_1$ into $e_1+e_2$ by the
operation $\sharp^\copyright$, so that one can also compress $F$
into an edge along $F_3$.

 \vskip .2cm (3) Suppose that
$F$ is a 3-independent pentagonal facet. Then up to $\text{GL}(3,
{\Bbb Z}_2)$-equivalence,  all possible colorings of $F$ and its
five neighboring facets may be listed as follows:
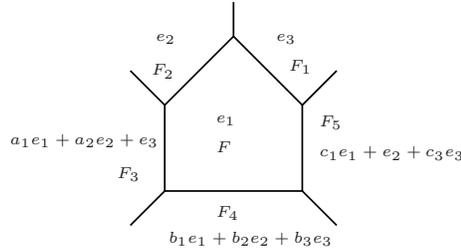
\[   \input{pent1.pstex_t}\centering
   \]
where $a_i, b_j, c_k\in {\Bbb Z}_2$ with $\begin{cases}
a_2b_3+b_2=1\\
b_3+b_2c_3=1
\end{cases}$ and at least one of $a_1, b_1$, and $c_1$ is nonzero.
An easy observation shows that if none of the five neighboring
facets around $F$ is  triangular, then there is always at least one
neighboring facet $F'$ around  $F$  so that one may compress $F$
into a $V_{eve}$ along $F'$. If there are triangular neighboring
facets around $F$, then the Steinitz's Theorem makes sure that the
number $\ell$ of such triangular neighboring facets is at most 2.

\vskip .2cm When the number $\ell$ is just 2, there are two
possibilities. With no loss assume that either $F_1$ and $F_3$ or
$F_1$ and $F_4$ are triangular. If  $F_1$ and $F_4$ are triangular,
then it is easy to check that at least one of $F_1$ and $F_4$ is
3-independent, so that one can compress the 3-independent triangular
facet into a point and then $F$ will become a rectangular facet.
Thus, the problem is reduced to the case (2) above. If  $F_1$ and
$F_3$ are triangular, then there are two possible cases: either  at
least one of $F_1$ and $F_3$ is 3-independent or both $F_1$ and
$F_3$ are 2-independent. If the former case happens, in a same way
as above, one may reduce this case to the case (2). If the latter
case happens,  by changing the coloring of $F_1$ or $F_3$ via the
operation $\sharp^\copyright$ (if necessary), one may compress $F$
into a $V_{eve}$ along $F_1$ or $F_3$.

\vskip .2cm When the number $\ell$ is just 1, with no loss assume
that $F_1$ is triangular. If $F_1$ is 3-independent, then one may
compress $F_1$ into a point, so that the case is reduced to the case
(2). If $F_1$ is 2-independent, then an easy argument shows that, by
adjusting the coloring of $F_1$ via the operation
$\sharp^\copyright$ (if necessary),  one may always compress $F$
into a $V_{eve}$ along $F_1$ or $F_3$ or $F_4$.
\end{proof}

\begin{rem}
We see from the proof of Proposition~\ref{compress} that if $F$ is
rectangular or pentagonal, then the  manner of compressing $F$  not
only depends upon the colorings of the  neighboring facets around
$F$, but also the existence of the triangular neighboring facets
around $F$. In general, the compression at $F$ may need two possible
steps: first compress the neighboring triangular facets around $F$,
and then perform the compression of $F$.  So after the compressions,
$F$ may  be compressed into an edge or a point  if $F$ is
rectangular, and a $V_{eve}$ or an edge or a point if $F$ is
pentagonal.
\end{rem}

\section{Proof of Theorem~\ref{T1}} \label{sec3}

Let $(P^3, \lambda)$ be a pair in $\mathcal{P}$. We shall finish the
proof of Theorem~\ref{T1} by using the descending induction on the
number of facets of the simple $3$-polytope $P^3$. Without the loss
of generality, assume that $P^3$ contains big facets.

\vskip .2cm

First, by Proposition~\ref{compress},  we can compress all  possible
3-independent small facets until we can not find them anymore. This
does not increase the number of facets of $P^3$. Let $(P_c^3,
\lambda_c)$ be the compression of $(P^3, \lambda)$ under this step,
and assume that $(P_c^3, \lambda_c)$ still contains big facets. Then
we divide our argument into two cases:

 \begin{enumerate}
\item[(A)]  there are adjacent
   big facets in $P^3_c$;
   \item[(B)]  there are no adjacent
   big facets in $P^3_c$.
\end{enumerate}
\vskip .2cm

{\bf Case (A).} Suppose that  there are adjacent
   big facets in $P^3_c$. Then there must be  a pair of adjacent big facets such that
   there is an adjacent small facet  as shown in the
   following
   picture:
    \[   \input{Adjacent-big-facet.pstex_t}\centering
   \]
This is because  the facets of $P^3_c$ are not all big according to
the Euler characteristic of $\partial P^3_c$.  Next we try to do the
equivariant Dehn surgery $\natural$ on $(P_c^3, \lambda_c)$.

\vskip .2cm

 When $C$ and $D$ have the same coloring,  we can do Dehn surgery $\natural$ on $(P_c^3,
 \lambda_c)$, which would reduce the number of facets by one. If this operation doesn't destroy the 3-connectedness of 1-skeleton of
     $P_c^3$, then we go on with our induction. Or else, by
     Proposition~\ref{composition} we have
     that $(P_c^3, \lambda_c)$
     can be separated into two smaller pairs $(P_1^3, \lambda_1)$ and $(P_2^3, \lambda_2)$ such that either
     $(P_c^3, \lambda_c)=(P_1^3, \lambda_1)\sharp^v(P_2^3, \lambda_2)$ or $(P_c^3, \lambda_c)=(P_1^3, \lambda_1)\sharp^\triangle(P_2^3, \lambda_2)$.
     Then the problem is reduced to carrying out  our inductions on  $(P_1^3, \lambda_1)$ and $(P_2^3, \lambda_2)$.

     \vskip .2cm

 When  $C$ and $D$ have different colorings, since the local coloring around $D$ is $2$-independent,
   by the operation $\sharp^\copyright$  we can change the coloring of $D$ to match the coloring of $C$. Then we can do the
     Dehn surgery operation, turning back to the above case.

     \vskip .2cm

     The above procedure can always be carried out until we can not find adjacent
   big facets anywhere. 
     \vskip .2cm

     {\bf Case (B).} If  there are no adjacent
   big facets in $P^3_c$, then any big facet is surrounded
   by 2-independent small facets. By the operation $\sharp^\copyright$, we can change the coloring of a small
   facet, say $F$, then the  adjacent small facets around $F$ become
    $3$-independent. Then we can compress them by using  operations $\sharp^v$, $\sharp^e$
      and $\sharp^{eve}$. We note that the edge number of the big
      facet
   will be reduced while we compress its neighboring triangular   facets,
   and this number will be either reduced or unchanged while we compress its neighboring  rectangular  facets,
   but this number  will be unchanged or reduced or becoming bigger
   while we compress its neighboring pentagonal   facets, as shown in the following
   figure:
   \[   \input{f8.pstex_t}\centering
   \]
   In particular, when we compress 3-independent pentagons,  it is possible to produce  new big
   facets. For example, if $F_1$ is a pentagon  in the above
   figure, then it will become a big facet after compressing $S$. In
   addition, it  is easy to see that compressing rectangles and
   pentagons may lead to the adjacency of big facets. If
   this happens, we can return to the case (A) to do Dehn surgeries.
   Otherwise, by changing the colorings of small facets and compressing them, we can carry on our
   work to reduce the edge numbers of big facets.

\vskip .2cm

  These alternate
   processes above can always end in finite steps  since the number of facets of $P^3$ is finite. Note that with the help of  $\sharp^\copyright$,
     five   operations $\sharp^v$, $\sharp^e$,  $\sharp^{eve}$, $\natural$,
   $\sharp^\triangle$ can not only decrease the number of facets, but they  can also decrease the edge numbers of big facets .

   \vskip .2cm
 On the other hand, we see that the six operations themselves also involve some
special colored blocks. Specifically, when doing three operations
$\sharp^v$, $\sharp^e$ and $\sharp^{eve}$, the colored $\Delta^3$,
$P^3(3)$ and $P^3_-(3)$ are involved; when doing the operation
$\natural$, the pair $(\text{\Large $\oslash$}, \tau)$ is involved.
As seen above,  using the equivariant Dehn surgery $\natural$, we
can avoid changing the colorings of big facets. So the special
colored blocks  involved in the operation $\sharp^\copyright$ are
only those colored $i$-sided prisms $P^3(i)$
 with   top and bottom facets
differently colored and being both 2-independent where $i=3,4,5$.

   \vskip .2cm

 Combining Cases (A) and (B),   we can always finish our induction by using the six types of operations until we reach
one of those  colored 3-polytopes $\Delta^3$, $P^3_-(3)$, and
$P^3(i)$ $(i=3,4,5)$ above. Furthermore, 
by reversing the induction process, eventually $(P^3, \lambda)$ can
be described as an expression of those colored blocks $\Delta^3$,
$P^3_-(3)$, $P^3(i)$ $(i=3,4,5)$, and $\text{\Large $\oslash$}$
under the six operations.

\vskip .2cm

Next, to complete the proof, let us make a further analysis for
those colored blocks $\Delta^3$, $P^3_-(3)$,  $P^3(i)$ $(i=3,4,5)$,
and $\text{\Large $\oslash$}$.

\vskip .2cm

First it is easy to see that $\Delta^3$ and $\text{\Large
$\oslash$}$ admit a unique $(\mathbb{Z}_2)^3$-coloring up to
$\text{\rm GL}(3,\mathbb{Z}_2)$-equivalence, as shown in the
following figure:
 \[   \input{fj1.pstex_t}\centering
   \]
   In particular, we have that $(\text{\Large $\oslash$}, \tau)=(\Delta^3, \lambda_0)\sharp^\triangle
   (\Delta^3, \lambda_0)$,  and this procedure is shown as follows:
    \[   \input{fj2.pstex_t}\centering
   \]
We know from \cite[Theorem 3.1]{CaiChenLu} that $P^3(3)$ admits five
kinds of colorings up to $\text{\rm
GL}(3,\mathbb{Z}_2)$-equivalence, which are listed as follows:
 \[   \input{f11.pstex_t}\centering
   \]
 Obviously, the colored $3$-sided prism on the right is the connected sum
         of two colored 3-simplices, as shown above. Now let
         us show that a colored $P^3_-(3)$ or $P^3(4)$ or $P^3(5)$
         can be obtained  from colored 3-simplices and 3-sided
         prisms via only two operations $\sharp^v$ and $\sharp^e$.
  \begin{enumerate}
        \item[(a)] From Figures (A)-(D)
of Remark~\ref{local coloring}(2), we can obtain that $P_-^3(3)$
admits nine kinds of
        colorings up to $\text{\rm GL}(3,\mathbb{Z}_2)$-equivalence, as listed in the following
        figure:
        \[   \input{f10.pstex_t}\centering
   \]
   We claim that up to $\text{\rm GL}(3,\mathbb{Z}_2)$-equivalence, by  doing the operation $\sharp^v$ of colored
         $P^3(3)$'s with colored $\Delta^3$'s, we can obtain the required nine kinds
         of colorings on $P^3_-(3)$.  The argument is not difficult. See the following figure for two special cases, and
         all the other cases are similar to these two.
        Notice that the connected sums
         $\sharp^v$ of a colored  $\Delta^3$ with some vertex of the
         top facet of a colored $P^3(3)$ and with some vertex of the bottom facet of a colored $P^3(3)$ respectively may produce
         different colorings of $P^3_-(3)$.
        \[   \input{paste.pstex_t}\centering   \]

        \item[(b)] 
        Consider the colored prisms $P^3(4)$ and $P^3(5)$ with
2-independent top and bottom facets differently colored. It is easy
to see that up to $\text{\rm GL}(3,\mathbb{Z}_2)$-equivalence and an
automorphism $h$ of $P^3(i)$, $P^3(4)$ admits six such colorings,
and $P^3(5)$ admits three such colorings,  where $h$ is the
automorphism of rotating  facets on the side. We list them as
follows:
 \[   \input{prism.pstex_t}\centering
   \]
    Note that clearly $h$ has no influence
   on the reconstruction
   of the above colored 3-polytopes up to equivariant homeomorphisms (cf \cite{CaiChenLu}).  Similarly  to the case (a), an easy argument shows
    that each of colored 4-sided prisms shown in
   Figures (P) and (Q) is the sum of two colored 3-sided prisms
   under the operation $\sharp^e$, and each of colored 5-sided
   prisms shown in Figure (R) is the sum of a colored 3-sided prism
   and a colored 4-sided prism under the operation $\sharp^e$  so it is also the sum of three colored 3-sided prisms
   under the operation $\sharp^e$.
      \end{enumerate}

\vskip .2cm

With all above arguments together, we see that,  up to $\text{\rm
GL}(3,\mathbb{Z}_2)$-equivalence there are only five elementary
colored 3-polytopes
as stated in Section 1, which can produce all
colored 3-polytopes under the six operations. This completes the
proof of Theorem~\ref{T1}. $\Box$

\begin{rem} \label{prism}
For a colored $m$-sided prism $(P^3(m), \lambda)$, since its facets
on the side are all squares,
 by considering 2-independence and 3-independence of  square facets,
        we can use operations $\sharp^e$ and $\sharp^\copyright$ alternately to compress facets on the side,
        so that $(P^3(m), \lambda)$ can be obtained from the
        colored $P^3(3)$ and $P^3(4)$. Since each colored 4-sided prism used in the operation $\sharp^\copyright$ above can be expressed as the sum of
        two
        colored $P^3(3)$'s under the operation $\sharp^e$ by the proof of Theorem~\ref{T1}, we
        conclude
                that $(P^3(m), \lambda)$ is a sum of some
        colored $P^3(3)$'s under the operations $\sharp^e$ and $\sharp^\copyright$.
\end{rem}

\section{Elementary 3-dimensional manifolds} \label{sec4}

The main task of this section is to determine those 3-dimensional
small covers corresponding to $(\Delta^3, \lambda_0)$ and $(P^3(3),
\lambda_i), i=1,2,3,4$, as stated in Section 1.

\vskip .2cm

Recall that  a small cover $\pi:M\longrightarrow P$ is equivariantly
homeomorphic to its reconstruction $M(P, \lambda)$ where the pair
$(P, \lambda)$ is determined by $M$. It is well-known (see \cite{DJ}
and \cite{Lu-cobordism}) that $n$-dimensional real projective space
$\mathbb{R}P^n$ admits a canonical linear $(\mathbb{Z}_2)^n$-action
defined by
$$[x_0,x_1,...,x_n]\longmapsto [x_0,g_1x_1,...,g_nx_n]$$
where $(g_1,...,g_n)\in (\mathbb{Z}_2)^n$. This action fixes $n+1$
fixed points $[\underbrace{0,...,0}_i,1,0,...,0]$, $i=0,1,...,n$,
and its orbit space is homeomorphic to the image of the  map $\Phi:
\mathbb{R}P^n\longrightarrow \mathbb{R}^{n+1}$ by
$$\Phi([x_0,x_1,...,x_n])=\Big({{|x_0|}\over{\sum_{i=0}^n|x_i|}},
{{|x_1|}\over{\sum_{i=0}^n|x_i|}}, ...,
{{|x_n|}\over{\sum_{i=0}^n|x_i|}}\Big).$$ It is easy to see that the
image of $\Phi$ is an $n$-dimensional simplex. A direct observation
shows that the $n+1$ facets of this $n$-simplex are colored by $e_1,
..., e_n, e_1+\cdots+e_n$ respectively, where $\{e_1,...,e_n\}$ is
the standard basis of $(\mathbb{Z}_2)^n$. This gives

\begin{lem} \label{basic small-1}
$M(\Delta^3, \lambda_0)$ is equivariantly homeomorphic to the
$\mathbb{R}P^3$ with a canonical linear $(\mathbb{Z}_2)^3$-action.
\end{lem}

The product of $\mathbb{R}P^1=S^1$ and $\mathbb{R}P^2$ with
canonical linear actions gives a canonical $(\mathbb{Z}_2)^3$-action
(denoted by $\phi_1$) on $S^1\times\mathbb{R}P^2$, which has exactly
six fixed points. Explicitly, this action on the product
$S^1\times\mathbb{R}P^2$ is defined by $$ \Big((g_1,g_2,g_3),
\big((x_0,x_1), [y_0,y_1,y_2]\big)\Big)\longmapsto
\big((x_0,g_1x_1), [y_0, g_2y_1,g_3y_2]\big).$$ The orbit space of
this action on $S^1\times\mathbb{R}P^2$ is the product of a
1-simplex and a 2-simplex, so it is just a 3-sided prism. It is also
easy to see that the orbit space of this action admits the same
coloring as $(P^3(3), \lambda_1)$. Thus we have
\begin{lem} \label{basic small-2}
$M(P^3(3), \lambda_1)$ is equivariantly homeomorphic to the product
 $S^1\times\mathbb{R}P^2$ with the canonical linear action
$\phi_1$.
\end{lem}

Regard $S^1$ as the unit circle $\{ z\in \C\ \big| |z|=1 \}$ in $\C$
and $\R P^2$ as the projective plane $\R P(\C\oplus \R) = \{\,
[v,w]\big| v\in \C, w\in \R\,
      \}$ in $\C\oplus \R$, we then construct
three $(\mathbb{Z}_2)^3$-actions $\phi_2, \phi_3,\phi_4$ on
$S^1\times \mathbb{R}P^2$ as follows:

 \begin{enumerate}
\item[(a)] The action $\phi_2$ on $S^1\times\mathbb{R}P^2$ is
defined by the following three commutative involutions
 \begin{align*}
           t_1: (z,[v,w]) &\longmapsto (\bar{z}, [zv,w]) \\
            t_2: (z,[v,w]) &\longmapsto (z, [-\bar{z}\bar{v},w])\\
             t_3: (z,[v,w]) &\longmapsto (\bar{z}, [-zv,w]).
        \end{align*}
        \item[(b)] The action $\phi_3$ on $S^1\times\mathbb{R}P^2$ is
defined by the following three commutative involutions
 \begin{align*}
           t_1: (z,[v,w]) &\longmapsto (\bar{z}, [zv,w]) \\
            t_2: (z,[v,w]) &\longmapsto (z, [\bar{z}\bar{v},w])\\
             t_3: (z,[v,w]) &\longmapsto (\bar{z}, [-zv,w]).
        \end{align*}
        \item[(c)] The action $\phi_4$ on $S^1\times\mathbb{R}P^2$ is
defined by the following three commutative involutions
 \begin{align*}
           t_1: (z,[v,w]) &\longmapsto (\bar{z}, [\bar{z}v,w]) \\
            t_2: (z,[v,w]) &\longmapsto (z, [z\bar{v},w])\\
             t_3: (z,[v,w]) &\longmapsto (\bar{z}, [-z\bar{v},w]).
        \end{align*}
                 \end{enumerate}

Note that the action $\phi_4$ was first given in
        \cite{Lu-cobordism}.         These three actions fix the same six  points
         $(\pm1, [1,0])$, $(\pm 1, [\text{\bf i}, 0])$ and $(\pm 1,
         [0,1])$, where $\text{\bf i}=\sqrt{-1}$.

\begin{lem}\label{basic small-3}
$M(P^3(3), \lambda_i), i=2,3,4,$ are equivariantly homeomorphic to
$(S^1\times\mathbb{R}P^2, \phi_i)$ respectively.
\end{lem}
\begin{proof}
First, let us show that each orbit space of the three actions is
homeomorphic to a 3-sided prism $P^3(3)$. For $z\in S^1$ and
$v\in\C$, write $z=e^{2\pi t\text{\bf i}}$ and
$v=re^{\theta\text{\bf i}}$ where $t\in [0,1]$, $r\in\R_{\geq 0}$,
and $\theta\in [0,2\pi]$. Then we define the map $\Phi: S^1\times
\mathbb{R}P^2\longrightarrow \mathbb{R}^5$ by
$$\Phi(z,[v,w])=(x_1, x_2, x_3, x_4, x_5)
$$
where
 \begin{align*}
  x_1&={{|\cos(2\pi t)|}\over{|\cos(2\pi t)|+|\sin(2\pi t)|}},\ \  x_2={{|\sin(2\pi t)|}\over{|\cos(2\pi t)|+|\sin(2\pi t)|}}, \\
    x_3&={{r|\cos(2\pi t+\theta)|}\over{r|\cos(2\pi t+\theta)|+r|\sin(2\pi
t+\theta)|+|w|}}, \\ x_4&={{r|\sin(2\pi t+\theta)|}\over{r|\cos(2\pi
t+\theta)|+r|\sin(2\pi t+\theta)|+|w|}},\\
x_5&={{|w|}\over{r|\cos(2\pi t+\theta)|+r|\sin(2\pi
t+\theta)|+|w|}}.
  \end{align*}
  Notice that $\cos[2\pi(1-t)+\theta]=\cos(2\pi t-\theta)$ and $|\sin[2\pi(1-t)+\theta]|=|\sin(2\pi
  t-\theta)|$.
Obviously, this map $\Phi$ is compatible with three actions $\phi_2,
\phi_3,\phi_4$ on $S^1\times \mathbb{R}P^2$. In particular,  we
easily see that for each  $t\in [0,1]$, the image of $\Phi$
restricted to $\mathbb{R}P^2$ is a 2-simplex, which consists of all
triples $(x_3, x_4, x_5)$. Also, the set $\{(x_1,x_2)\big| t\in
[0,1]\}$ forms a 1-simplex. Thus, the image of $\Phi$ is a 3-sided
prism. Furthermore, it is easy to see that each orbit space of the
three actions is homeomorphic to this 3-sided prism.

\vskip .2cm

Next we show that the orbit space of the action $\phi_i$ admits the
same coloring as $(P^3(3), \lambda_i)$. We shall only consider the
case $i=2$ because the arguments of other two cases are similar. Our
strategy is to first  determine the tangent representations at those
fixed points and then to give the coloring on the orbit space by
using algebraic duality.

\vskip .2cm

  $\text{\rm Hom}((\mathbb{Z}_2)^3,
\mathbb{Z}_2)$, which consists all homomorphism from
$(\mathbb{Z}_2)^3$ to $\mathbb{Z}_2$, gives all irreducible
representations of $(\mathbb{Z}_2)^3$, and forms an abelian group
with addition given by $(\rho+\eta)(g)=\rho(g)\eta(g)$, where $g\in
(\mathbb{Z}_2)^3$. The homomorphisms $\rho_j:
g=(g_1,g_2,g_3)\longmapsto g_j, j=1,2,3,$ form a basis of $\text{\rm
Hom}((\mathbb{Z}_2)^3, \mathbb{Z}_2)$. Now write $v=(v_1, v_2)$.
When $z=-1$,  the action $\phi_2$ restricted to $\{-1\}\times
\mathbb{R}P^2$ can be defined by the following way
 \begin{align*}
 \big(g, (-1, [v_1,v_2,w])\big)\longmapsto &\big(-1,
[\rho_1(g)\rho_2(g)v_1, \rho_1(g)\rho_2(g)\rho_3(g)v_2, w]\big)\\
=&\big(-1, [\rho_3(g)v_1, v_2,
\rho_1(g)\rho_2(g)\rho_3(g)w]\big)\\
=&\big(-1, [v_1, \rho_3(g)v_2, \rho_1(g)\rho_2(g)w]\big)
\end{align*}
and when $z=1$, the action $\phi_2$ restricted to $\{1\}\times
\mathbb{R}P^2$ can be defined by the following way
\begin{align*}
 \big(g, (1, [v_1,v_2,w])\big)\longmapsto &\big(1,
[\rho_2(g)v_1, \rho_2(g)\rho_3(g)v_2, w]\big)\\
=&\big(1, [\rho_3(g)v_1, v_2,
\rho_2(g)\rho_3(g)w]\big)\\
=&\big(1, [v_1, \rho_3(g)v_2, \rho_2(g)w]\big)
\end{align*}
Then we can read off the tangent representations at six fixed
points, which determine a $\text{\rm Hom}((\mathbb{Z}_2)^3,
\mathbb{Z}_2)$-coloring on 1-skeleton of the orbit space by GKM
theory (see \cite{GKM98} and \cite{L}), as shown in the following
figure:
  \[   \input{duality.pstex_t}\centering
   \]
   This $\text{\rm Hom}((\mathbb{Z}_2)^3,
\mathbb{Z}_2)$-coloring is dual to the $(\mathbb{Z}_2)^3$-coloring
on the orbit space by $\rho_i(e_j)=\begin{cases} 1& \text{ if }
i=j\\
0 & \text{ if } i\not=j \end{cases}$ (cf \cite[Proposition 4.1]
{Lu_Graphs}), so we can obtain the desired coloring, as shown in the
above figure.
\end{proof}

Although $\text{\Large $\oslash$}$ is not a 3-polytope,  it is
contractible, so we can apply the method of reconstruction of small
covers to
   $(\text{\Large $\oslash$},\tau)$ to obtain a 3-manifold, denoted by $M(\text{\Large $\oslash$},\tau)$.

\begin{lem}\label{sphere}
$M(\text{\Large $\oslash$}, \tau)$ is equivariantly homeomorphic to
the $S^3$ with the standard $(\mathbb{Z}_2)^3$-action. Moreover, so
is $M(\Delta^3, \lambda_0)\widetilde{\sharp^\triangle}M(\Delta^3,
\lambda_0)$.
\end{lem}

\begin{proof} Consider the standard $(\mathbb{Z}_2)^3$-action on $S^3$ by
$$(x_0,x_1,x_2,x_3)\longmapsto (x_0,g_1x_1,g_2x_2,g_3x_3).$$
Obviously, this action has  two fixed points $(\pm 1,0,0,0)$, and
its orbit space is identified with $\text{\Large $\oslash$}$. A
direct observation shows that  three 2-polygon faces of the orbit
space are colored by $e_1, e_2, e_3$, so this agrees with the
coloring $\tau$ on $\text{\Large $\oslash$}$. Since $\text{\Large
$\oslash$}$ is contractible, any principal $(\mathbb{Z}_2)^3$-bundle
over $\text{\Large $\oslash$}$ is trivial. Furthermore, by the
method of reconstruction it is easy to see that $M(\text{\Large
$\oslash$}, \tau)$ is equivariantly homeomorphic to the $S^3$ with
the standard $(\mathbb{Z}_2)^3$-action.\end{proof}

\begin{rem}
We easily see from \cite[Theorem 3.1]{DJ} that $M(\text{\Large
$\oslash$},\tau)$ is not a small cover. In fact, any $n$-sphere
$S^n$ with $n>1$ can not become a small cover. This is because its
mod 2 Betti numbers $(1,0,...,0,1)$ can not be used as the
$h$-vector of any simple convex $n$-polytope. But $S^1$ is a small
cover. Also, it is easy to see that both $M(\text{\Large
$\oslash$},\tau)$ and $M(\text{\Large $\oslash$},\sigma\circ\tau)$
are $\sigma$-equivariantly homeomorphic, where $\sigma\in \text{\rm
GL}(3,\Z)$.
\end{rem}

By the reconstruction of small covers, together with
Theorem~\ref{T1} and Lemmas~\ref{basic small-1}, \ref{basic
small-2},
 \ref{basic small-3} and~\ref{sphere}, we have completed the proof of
Theorem~\ref{T2}. It remains to understand the geometrical meanings
of corresponding six operations on $\mathcal{M}$.

\section{Operations on $\mathcal{M}$}\label{sec5}

Now let us look at how corresponding six operations  work on
$\mathcal{M}$. In particular, this will tell us how to construct a
small cover 3-manifold by using cut
  and paste strategies.

  \vskip .2cm

 To understand six operations on
$\mathcal{M}$, first let us study the corresponding geometrical
meanings of sections $S_v, S_e, S_{V_{eve}}$,  $S_\triangle,
S_\square, S_\maltese$ in small covers. These sections actually
correspond to some closed surfaces, which we list in the following
lemma.

\begin{lem}\label{section}
The corresponding geometrical meanings (up to homeomorphism) of
sections $S_v, S_e, S_{V_{eve}}$,  $S_\triangle, S_\square,
S_\maltese$ in small covers are stated as follows:
\begin{enumerate}
\item $S_v$ corresponds to a 2-sphere $S^2$;

\item $S_e$  corresponds to a 2-dimensional torus $T$ or a Klein bottle $K$ shown as follows:
 \[   \input{section-1.pstex_t}\centering
   \]

\item $S_{V_{eve}}$  corresponds to a $T\# T$ or a $K\# K$ shown as follows:
 \[   \input{section-2.pstex_t}\centering
   \]

\item $S_\triangle$  corresponds to a disjoint union $\mathbb{R}P^2\sqcup\mathbb{R}P^2$;

\item $S_\square$  corresponds to a $T\sqcup T$ or a $K\sqcup K$ shown as follows:
 \[   \input{section-3.pstex_t}\centering
   \]

\item $S_\maltese$  corresponds to a disjoint union
$(\mathbb{R}P^2\#\mathbb{R}P^2\#\mathbb{R}P^2)\sqcup
(\mathbb{R}P^2\#\mathbb{R}P^2\#\mathbb{R}P^2)$,
\end{enumerate}
where $\#$ denotes the ordinary connected sum.
\end{lem}

\begin{proof}
The argument is not quite difficult, and it is  mainly based upon
the reconstruction method of small covers. We would like to leave it
to readers as an exercise.
\end{proof}

\begin{rem}
Lemma~\ref{section} will play a beneficial role in understanding the
six operations on $\mathcal{M}$. It should be pointed out that  the
corresponding closed surfaces  of those sections are all not small
covers in the sense of Davis-Januszkiewicz. Actually, for each such
section $S$, its corresponding closed 2-manifold $M^2$ is the double
covering space of a small cover over $S$. Also,  if $S$ is
2-independent then $M^2$ is disconnected, and  if $S$ is
3-independent then $M^2$ is connected.
\end{rem}

  \subsection{Operation $\widetilde{\sharp^v}$ on $\mathcal{M}$}
  This operation is actually the equivariant connected sum. By Lemma~\ref{section}, cutting out a vertex $v$ of a
  colored $(P^3, \lambda)$ exactly corresponds to cutting out a
  $(\mathbb{Z}_2)^3$-invariant
  open 3-ball which contains a fixed point of $M(P^3, \lambda)$ as shown in the following figure, so
  that the operation $\sharp^v$ on $\mathcal{P}$ induces the
  equivariant connected sum $\widetilde{\sharp^v}$ on $\mathcal{M}$.
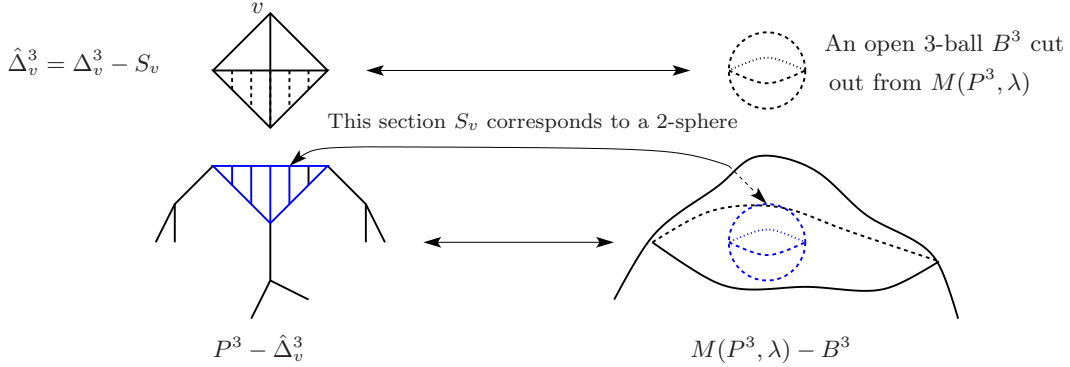
 \[   \input{operation1.pstex_t}\centering
   \]

Now from the proof of Theorem~\ref{T1} and by Lemmas~\ref{basic
small-1}, \ref{basic small-2}, \ref{basic small-3}
and~\ref{section}, we have

\begin{cor} \label{top-type1}
  The topological type of $M(P^3(3),
\tau)$ is either $\mathbb{R}P^3\#\mathbb{R}P^3$ or $S^1\times
\mathbb{R}P^2$. Furthermore, the topological type of $M(P^3_-(3),
\tau)$ is either $\mathbb{R}P^3\#\mathbb{R}P^3\#\mathbb{R}P^3$ or
$(S^1\times \mathbb{R}P^2)\#\mathbb{R}P^3$.
\end{cor}

\subsection{Operation $\widetilde{\sharp^e}$ on $\mathcal{M}$}
By Lemma~\ref{section}, when we do the operation
$\widetilde{\sharp^e}$ on a $M(P^3, \lambda)$, we exactly cut out a
  $(\mathbb{Z}_2)^3$-invariant
open solid torus $\widehat{T}$ (or a
  $(\mathbb{Z}_2)^3$-invariant open solid Klein bottle
$\widehat{K}$) from $M(P^3,\lambda)$, while we also need to cut out
a same type of
  $(\mathbb{Z}_2)^3$-invariant
open solid torus (or a same type of
  $(\mathbb{Z}_2)^3$-invariant open solid Klein bottle)  from a $M(P^3(3),
\tau)$. However, by Corollary~\ref{top-type1} $M(P^3(3), \tau)$ has
two different   topological types: either
$\mathbb{R}P^3\#\mathbb{R}P^3$ or $S^1\times \mathbb{R}P^2$.
According to the colorings on $P^3(3)$, an easy argument shows that
when the topological type of $M(P^3(3), \tau)$ is
$\mathbb{R}P^3\#\mathbb{R}P^3$, we can only cut out a
  $(\mathbb{Z}_2)^3$-invariant
open solid torus from $M(P^3(3), \tau)$, but when the topological
type of $M(P^3(3), \tau)$ is  $S^1\times \mathbb{R}P^2$, we can not
only cut out a
  $(\mathbb{Z}_2)^3$-invariant open solid torus  but also a
  $(\mathbb{Z}_2)^3$-invariant open solid Klein bottle from
$M(P^3(3), \tau)$. More precisely, up to $\text{\rm
GL}(3,\mathbb{Z}_2)$-equivalence, when $\tau=\lambda_i, i=1, 4$, we
can only cut out a
  $(\mathbb{Z}_2)^3$-invariant open solid torus from $M(P^3(3), \lambda_i)$ and when
$\tau=\lambda_i,i=2,3$, we can only cut out a
  $(\mathbb{Z}_2)^3$-invariant open solid Klein
bottle from $M(P^3(3), \lambda_i)$.
 Thus, we have that if the topological type of $M(P^3(3), \tau)$ is
$\mathbb{R}P^3\#\mathbb{R}P^3$, then
$$M(P^3,
\lambda)\widetilde{\sharp^e}M(P^3(3), \tau)= \big(M(P^3,
\lambda)\backslash\widehat{T}\big)\cup_T\big(M(P^3(3),
\tau)\backslash\widehat{T}\big)$$ and if the topological type of
$M(P^3(3), \tau)$ is  $S^1\times \mathbb{R}P^2$, then
$$M(P^3,
\lambda)\widetilde{\sharp^e}M(P^3(3), \tau)=\begin{cases}
\big(M(P^3,
\lambda)\backslash\widehat{T}\big)\cup_T\big(M(P^3(3), \tau)\backslash\widehat{T}\big) & \text{ if } \tau=\lambda_1, \lambda_4\\
 \big(M(P^3,
\lambda)\backslash\widehat{K}\big)\cup_K\big(M(P^3(3),
\tau)\backslash\widehat{K}\big) & \text{ if } \tau=\lambda_2,
\lambda_3.
\end{cases}$$

\subsection{Operation $\widetilde{\sharp^{eve}}$ on $\mathcal{M}$}
   Similarly, by Lemma~\ref{section}, when we do the operation
$\widetilde{\sharp^{eve}}$ on a $M(P^3, \lambda)$, we need to cut
out a same type of
  $(\mathbb{Z}_2)^3$-invariant $\widehat{T\# T}$ (or a same type of
  $(\mathbb{Z}_2)^3$-invariant $\widehat{K\# K}$) from
$M(P^3,\lambda)$ and $M(P^3_-(3),\tau)$ respectively, and then glue
the remaining parts together along their boundaries, where
$\widehat{T\# T}$ (resp. $\widehat{K\# K}$) denotes the interior of
a 3-dimensional $(\mathbb{Z}_2)^3$-manifold with boundary $T\# T$
(resp. $K\# K$). We know from Corollary~\ref{top-type1} that the
topological type of $M(P^3_-(3),\tau)$ is either
$\mathbb{R}P^3\#\mathbb{R}P^3\#\mathbb{R}P^3$ or $(S^1\times
\mathbb{R}P^2)\#\mathbb{R}P^3$. According to the colorings on
$P^3_-(3)$, we see  easily  that if the topological type of
$M(P^3_-(3),\tau)$ is $\mathbb{R}P^3\#\mathbb{R}P^3\#\mathbb{R}P^3$,
then we can only cut out a $(\mathbb{Z}_2)^3$-invariant
$\widehat{T\# T}$ from $M(P^3_-(3),\tau)$, and if the topological
type of $M(P^3_-(3),\tau)$ is $(S^1\times
\mathbb{R}P^2)\#\mathbb{R}P^3$, we can only cut out a
$(\mathbb{Z}_2)^3$-invariant $\widehat{K\# K}$ from
$M(P^3_-(3),\tau)$. Therefore, we have that when the topological
type of $M(P^3_-(3),\tau)$ is
$\mathbb{R}P^3\#\mathbb{R}P^3\#\mathbb{R}P^3$,
$$M(P^3,
\lambda)\widetilde{\sharp^{eve}}M(P^3_-(3),\tau)= \big(M(P^3,
\lambda)\backslash\widehat{T\# T}\big)\cup_{T\#
T}\big(M(P^3_-(3),\tau) \backslash\widehat{T\# T}\big)$$ and when
the topological type of $M(P^3_-(3),\tau)$ is $(S^1\times
\mathbb{R}P^2)\#\mathbb{R}P^3$,
$$M(P^3,
\lambda)\widetilde{\sharp^{eve}}M(P^3_-(3),\tau)= \big(M(P^3,
\lambda)\backslash\widehat{K\# K}\big)\cup_{K\# K}\big(
M(P^3_-(3),\tau) \backslash\widehat{K\# K}\big).$$

\subsection{Operation $\widetilde{\natural}$ on $\mathcal{M}$}

Recall (cf \cite{Prasolov} and \cite{Rolfsen}) that a ${q\over
p}$-type Dehn surgery on a 3-manifold $M^3$ is as follows: removing
a solid torus from $M^3$ and then sewing it back in $M^3$ such that
the meridian goes to $p$ times the longitude and $q$ times the
meridian, where $p,q\in\mathbb{Z}$.

\vskip .2cm

{\em Claim.} The operation $\widetilde{\natural}$ on
 $M(P^3, \lambda)$ is exactly  an equivariant ${0\over 1}$-type Dehn
surgery on $M(P^3,\lambda)$.

\vskip .2cm
   In fact, when we cut out an edge from $(\text{\Large $\oslash$},\tau)$, the section is a 3-colorable square, so by
   Lemma~\ref{section} we exactly cut out a $(\mathbb{Z}_2)^3$-invariant open solid torus from
   $M(\text{\Large $\oslash$},\tau)$. On the other hand, using the method of the reconstruction of
small covers, the remaining part of the $(\text{\Large
$\oslash$},\tau)$ can be reconstructed into a
$(\mathbb{Z}_2)^3$-invariant  solid torus. So the operation
$\widetilde{\natural}$ will remove  a $(\mathbb{Z}_2)^3$-invariant
open solid torus $N_1$ from $M(P^3,\lambda)$ and glue back another
$(\mathbb{Z}_2)^3$-invariant  solid torus $N_2$ come from
$M(\text{\Large $\oslash$},\tau)$, mapping the meridian (longitude)
of $N_2$ to the longitude (meridian) of $N_1$. Notice that each edge
in $(P^3,\lambda)$ corresponds to a circle in $M(P^3,\lambda)$ by
the reconstruction of small covers.

\vskip .2cm

Therefore, the operation $\widetilde{\natural}$ on a $M(P^3,
\lambda)$ up to $\text{\rm GL}(3,\mathbb{Z}_2)$-equivalence can be
expressed as follows:

$$M(P^3, \lambda)\widetilde{\natural}M(\text{\Large $\oslash$},\tau)=
\big(M(P^3, \lambda)\backslash
\widehat{T}\big)\cup_T\big(M(\text{\Large $\oslash$},\tau)\backslash
\widehat{T}\big).$$

\subsection{Operation $\widetilde{\sharp^\triangle}$ on $\mathcal{M}$}
When we do the operation $\widetilde{\sharp^\triangle}$ on two
$M(P_1^3, \lambda_1)$ and $M(P_2^3, \lambda_2)$, since $S_\triangle$
corresponds to a disjoint union $\mathbb{R}P^2\sqcup \mathbb{R}P^2$
by Lemma~\ref{section}, we need to cut out a
$(\mathbb{Z}_2)^3$-invariant $\mathbb{R}P^2\times(-1,1)$ from each
of both $M(P_1^3, \lambda_1)$ and $M(P_2^3, \lambda_2)$. Then we
glue them together along their boundaries. Thus, we have
 \begin{align*}
 &M(P_1^3, \lambda_1)\widetilde{\sharp^\triangle}M(P_2^3, \lambda_2)\\
 =&
\big(M(P_1^3,
\lambda_1)\backslash(\mathbb{R}P^2\times(-1,1))\big)\cup_{\mathbb{R}P^2\sqcup\mathbb{R}P^2}
\big(M(P_2^3, \lambda_2)\backslash(\mathbb{R}P^2\times(-1,1))\big).
\end{align*}
Notice that the two $(\mathbb{Z}_2)^3$-invariant
$\mathbb{R}P^2\times(-1,1)$ cut out from   $M(P_1^3, \lambda_1)$ and
$M(P_2^3, \lambda_2)$ may not always be equivariantly homeomorphic
because we may cut out two  triangular facets with different
colorings from $(P_1^3, \lambda_1)$ and $(P_2^3, \lambda_2)$.

\subsection{Operation $\widetilde{\sharp^\copyright}$ on
$\mathcal{M}$} As we have seen, when we do the operation
$\sharp^\copyright$ on $\mathcal{P}$, only 2-independent small
facets are involved. Thus, when we do the operation
$\widetilde{\sharp^\copyright}$ on a $M(P^3,\lambda)$, 
by Lemma~\ref{section}(4)-(6) we need to cut out a
$(\mathbb{Z}_2)^3$-invariant $\mathbb{R}P^2\times(-1,1)$, or a
$(\mathbb{Z}_2)^3$-invariant  $T\times(-1,1)$, or a
$(\mathbb{Z}_2)^3$-invariant $K\times(-1,1)$, or a
$(\mathbb{Z}_2)^3$-invariant
$(\mathbb{R}P^2\#\mathbb{R}P^2\#\mathbb{R}P^2)\times(-1,1)$ from
$M(P^3,\lambda)$, and at the same time,  up to $\text{\rm
GL}(3,\mathbb{Z}_2)$-equivalence we also need to do same things on
$M(P^3(i), \tau), i=3,4,5$, where the top facet and the bottom facet
of each $(P^3(i),\tau)$ are colored by two different colors and the
colorings of neighboring facets around them are 2-independent, as
shown in Figures (H),(P)-(R) of Section~\ref{sec3}, then gluing
their corresponding boundaries together. When $i=3$, by
 Lemmas~\ref{basic small-2} and~\ref{basic
small-3}, the topological type of $M(P^3(3), \tau)$ is exactly
$S^1\times\mathbb{R}P^2$. When $i=4,5$, we know from the proof of
Theorem~\ref{T1} that $M(P^3(4), \tau)$ is the sum of two $M(P^3(3),
\eta_1)$ and $M(P^3(3), \eta_2)$ under $\widehat{\sharp^e}$, and
$M(P^3(5), \tau)$ is the sum of a $M(P^3(3), \eta)$ and a $M(P^3(4),
\kappa)$ under $\widehat{\sharp^e}$. However, this does not make
clear what the topological types of $M(P^3(4), \tau)$ and $M(P^3(5),
\tau)$ are. Next, we shall investigate their topological types.

\vskip .2cm It is well
   known  (see \cite{Lickorish63}) that for any closed surface $\Sigma$, $\Sigma$-bundles over $S^1$ are
   classified by the mapping class group $\text{\rm \bf MCG}^*(\Sigma)$. In
   particular, \begin{enumerate}
     \item[(I)] when $\Sigma$ is a torus $T$, $\text{\rm \bf MCG}^*(T)\cong \text{\rm SL}(2,\Z)=\text{\rm Aut}(H_2(T,\Z))$.\\

     \item[(II)] when $\Sigma$ is a Klein bottle $K$, $\text{\rm \bf MCG}^*(K)= \Z_2\oplus \Z_2$. In fact, if we think
     of $K$ as $S^1\times S^1/(z_1,z_2)\sim (-z_1,\bar{z}_2)$, then elements in
     $\text{\rm \bf MCG}^*(K)$ can be represented by
     $    \{ f_{\epsilon_1\epsilon_2}\big| \epsilon_1=\pm 1,
          \epsilon_2=\pm 1  \}
     $ where $f_{\epsilon_1\epsilon_2}([z_1,z_2]) = ([z_1^{\epsilon_1},
     z_2^{\epsilon_2}])$.
  \end{enumerate}

  First let us look at the three colored 4-sided prisms
 shown in Figure (P) of Section~\ref{sec3}, denoted by $(P^3(4),\tau_j), j=1,2,3$, respectively.

  \begin{lem}\label{4-prism1}
$M(P^3(4),\tau_j), j=1,2,3$, are equivariantly homeomorphic to
  three twisted $T$-bundles over $S^1$ with  monodromy maps $
         \begin{pmatrix}
          -1 & 0\\ 0 & 1
         \end{pmatrix}$,
         $\begin{pmatrix}
          1 & 0\\ 0 & -1
         \end{pmatrix}$,
        $\begin{pmatrix}
          -1 & 0\\ 0 & -1
         \end{pmatrix} \in \text{\rm \bf MCG}^*(T)$, respectively, where the $(\mathbb{Z}_2)^3$-action on each twisted $T$-bundle over $S^1$
         is induced by the $(\mathbb{Z}_2)^3$-action $\psi$ on $T\times[-1,1]$ defined by the following three
commutative involutions
\begin{align*}
     &  t_1:(z_1,z_2, t)\longmapsto (\bar{z}_1,z_2,t)\\
      & t_2:(z_1,z_2, t)\longmapsto (z_1,\bar{z}_2,t)\\
      & t_3:(z_1,z_2, t)\longmapsto (z_1,z_2,-t).
      \end{align*}
  \end{lem}
 \begin{proof}
By Lemma~\ref{section}, any horizontal
 section of each $(P^3(4), \tau_j)$ corresponds to a disjoint union $T\sqcup T$ in $M(P^3(4), \tau_j)$. This means
 that the  two parts obtained by cutting each $(P^3(4),\tau_j)$
 horizontally  correspond to two $(\mathbb{Z}_2)^3$-invariant
 $T$-handlebodies $T$-$HB_{j1}$ and $T$-$HB_{j2}$, each of which
 is homeomorphic to $T\times[-1,1]$, as shown in the following
 figure:
 \[   \input{4-prism.pstex_t}\centering
   \]
Obviously,  all $T$-$HB_{j1}$'s are equivariantly homeomorphic to
the $T\times[-1,1]$ with the $(\mathbb{Z}_2)^3$-action $\psi$. An
easy observation shows that   $T$-$HB_{j2}, j=1,2,3$, are obtained
from the $T\times[-1,1]$ with the $(\mathbb{Z}_2)^3$-action $\psi$
by using  the following Dehn twists on $T\times[-1,1]$
\begin{align*}
     &  d_1:(z_1,z_2, t)\longmapsto (e^{\pi (t+1)\text{\bf i}}z_1,z_2,t)\\
      & d_2:(z_1,z_2, t)\longmapsto (z_1,e^{\pi (t+1)\text{\bf i}}z_2,t)\\
      & d_3:(z_1,z_2, t)\longmapsto (e^{\pi (t+1)\text{\bf i}}z_1,e^{\pi (t+1)\text{\bf i}}z_2,t),
      \end{align*}
      respectively.
Namely, the topological types of  $T$-$HB_{j2} (j=1,2,3)$ are
\begin{align*}
     &  d_1(T\times[-1,1])=\{ (e^{\pi (t+1)\text{\bf i}}z_1,z_2,t)\big| z_1, z_2\in S^1, t\in [-1,1]\}\\
      &  d_2(T\times[-1,1])=\{ (z_1,e^{\pi (t+1)\text{\bf i}}z_2,t)\big| z_1, z_2\in S^1, t\in [-1,1]\}\\
      & d_3(T\times[-1,1])=\{ (e^{\pi (t+1)\text{\bf i}}z_1,e^{\pi (t+1)\text{\bf i}}z_2,t)\big| z_1, z_2\in S^1, t\in
      [-1,1]\}
      \end{align*}
      respectively, and they admit the $(\mathbb{Z}_2)^3$-actions
      which are compatible with the $(\mathbb{Z}_2)^3$-action $\psi$ on
      $T\times[-1,1]$, as follows:
\begin{enumerate}
\item[(i)] The
$(\mathbb{Z}_2)^3$-action $\psi_1$ on $d_1(T\times[-1,1])$ is given
by the following three commutative involutions
\begin{align*}
     &  t_1:(e^{\pi (t+1)\text{\bf i}}z_1,z_2, t)\longmapsto (e^{\pi (t+1)\text{\bf i}}\bar{z}_1,z_2,t)\\
      & t_2:(e^{\pi (t+1)\text{\bf i}}z_1,z_2, t)\longmapsto (e^{\pi (t+1)\text{\bf i}}z_1,\bar{z}_2,t)        \\
      & t_3:(e^{\pi (t+1)\text{\bf i}}z_1,z_2, t)\longmapsto (e^{\pi (t+1)\text{\bf
      i}}z_1,z_2,-t)
      \end{align*}
satisfying $\psi d_1=d_1\psi_1$.

\item[(ii)] The
$(\mathbb{Z}_2)^3$-action $\psi_2$ on $d_2(T\times[-1,1])$ is given
by the following three commutative involutions
\begin{align*}
     &  t_1:(z_1,e^{\pi (t+1)\text{\bf i}}z_2, t)\longmapsto (\bar{z}_1,e^{\pi (t+1)\text{\bf i}}z_2,t)\\
      & t_2:(z_1,e^{\pi (t+1)\text{\bf i}}z_2, t)\longmapsto (z_1,e^{\pi (t+1)\text{\bf i}}\bar{z}_2,t)\\
      & t_3:(z_1,e^{\pi (t+1)\text{\bf i}}z_2, t)\longmapsto (z_1,e^{\pi (t+1)\text{\bf i}}z_2,-t)
      \end{align*}
satisfying $\psi d_2=d_2\psi_2$.
\item[(iii)] The
$(\mathbb{Z}_2)^3$-action $\psi_3$ on $d_3(T\times[-1,1])$ is given
by the following three commutative involutions
\begin{align*}
     &  t_1:(e^{\pi (t+1)\text{\bf i}}z_1,e^{\pi (t+1)\text{\bf i}}z_2, t)\longmapsto (e^{\pi (t+1)\text{\bf i}}\bar{z}_1,e^{\pi (t+1)\text{\bf i}}z_2,t)\\
      & t_2:(e^{\pi (t+1)\text{\bf i}}z_1,e^{\pi (t+1)\text{\bf i}}z_2, t)\longmapsto (e^{\pi (t+1)\text{\bf i}}z_1,e^{\pi (t+1)\text{\bf i}}\bar{z}_2,t)\\
      & t_3:(e^{\pi (t+1)\text{\bf i}}z_1,e^{\pi (t+1)\text{\bf i}}z_2, t)\longmapsto (e^{\pi (t+1)\text{\bf i}}z_1,e^{\pi (t+1)\text{\bf i}}z_2,-t)
      \end{align*}
      satisfying $\psi d_3=d_3\psi_3$.
\end{enumerate}
When $t=\pm 1$, we have $e^{\pi (t+1)\text{\bf i}}=1$, so we see
that each $M(P^3(4),\tau_j)$ is obtained by equivariantly gluing
$T\times [-1,1]$ and $d_j(T\times[-1,1])$ along their boundaries via
the identity of $T$.  On the other hand, when $t=0$, we have $e^{\pi
(t+1)\text{\bf i}}=-1$, so we see that the three Dehn twists $d_1,
d_2, d_3$  determine exactly three monodromy maps  $\sigma_j:
T\longrightarrow T, j=1,2,3$, as follows:
\begin{align*}
     &  \sigma_1:(z_1,z_2)\longmapsto (z_1,z_2)\begin{pmatrix}
          -1 & 0\\ 0 & 1
         \end{pmatrix}=(-z_1,z_2)\\
      & \sigma_2:(z_1,z_2)\longmapsto (z_1,z_2)\begin{pmatrix}
          1 & 0\\ 0 & -1
         \end{pmatrix}=(z_1,-z_2)\\
      & \sigma_3:(z_1,z_2)\longmapsto (z_1,z_2)\begin{pmatrix}
          -1 & 0\\ 0 & -1
         \end{pmatrix}=(-z_1,-z_2).
      \end{align*}
This completes the proof.
  \end{proof}

  Let $(P^3(4),\tau_j), j=4,5,6$, denote those three colored 4-sided
  prisms shown in Figure (Q) of
Section~\ref{sec3}. In a similar way, we can prove the following
lemma.
 \begin{lem}\label{4-prism2}
 $M(P^3(4),\tau_j), j=4,5,6$, are equivariantly homeomorphic to
  three twisted $K$-bundles over $S^1$ with  monodromy maps
      $f_{-1,1}, f_{1,-1}$ and $f_{-1,-1}
        \in \text{\rm \bf MCG}^*(K)$
         respectively, where the $(\mathbb{Z}_2)^3$-action on each twisted $K$-bundle over $S^1$
         is induced by the $(\mathbb{Z}_2)^3$-action $\kappa$ on $K\times[-1,1]$ defined by the following three
commutative involutions
\begin{align*}
     &  t_1:([z_1,z_2], t)\longmapsto ([\bar{z}_1,z_2],t)\\
      & t_2:([z_1,z_2], t)\longmapsto ([z_1,\bar{z}_2],t)\\
      & t_3:([z_1,z_2], t)\longmapsto ([z_1,z_2],-t).
      \end{align*}
    \end{lem}

    Let
    $N = T_0 \cup_\partial M_0$ where $T_0$ is a punctured torus and $M_0$ is a M\"obius
    band with $T_0\cap M_0 = \partial T_0 = \partial M_0$. Then $N$ is homeomorphic to
    $\mathbb{R}P^2\#\mathbb{R}P^2\#\mathbb{R}P^2$. It is well
    known  (see \cite{BirChi}) that any diffeomorphism of $N$ is isotopic to one leaving
      $T_0$ and $M_0$ invariant, and there is the following result.

    \begin{lem}[\cite{BirChi}] The extended mapping class group $\text{\bf MCG}^*_+(N)$ of $N$ is isomorphic to
      $\text{\rm GL}(2,\Z)$, and the isomorphism is given by the natural homomorphism
      \[ \Pi: \text{\bf MCG}^*_+(N) \rightarrow \text{\rm Aut}\left( H_1(N;\Z)\slash
           \text{\rm Tor}(H_1(N;\Z))\right)=\text{\rm Aut}(H_1(T;\Z))
           \cong \text{\rm GL}(2,\Z)
      \]
      where $T=T_0\cup_\partial D^2$ is a torus.
\end{lem}

 Let $(P^3(5),\eta_j), j=1,2,3$, denote those three colored 5-sided
  prisms shown in Figure (R) of
Section~\ref{sec3}. Then we have

\begin{lem}\label{5-prism}
$M(P^3(5),\eta_j), j=1,2,3$, are equivariantly homeomorphic to three
special twisted $N$-bundles over $S^1$ with monodromy maps as the
inverse images of $\begin{pmatrix}
          -1 & 0\\ 0 & 1
         \end{pmatrix}$,
         $\begin{pmatrix}
          1 & 0\\ 0 & -1
         \end{pmatrix}$,
        $\begin{pmatrix}
          -1 & 0\\ 0 & -1
         \end{pmatrix} \in \text{\rm GL}(2,\Z)$ respectively under the isomorphism
         $\Pi$.
\end{lem}
\begin{proof}
In fact, each $(P^3(5),\eta_j)$ can be constructed by using a
colored 3-sided prism and a colored 4-sided prism under the
operation $\sharp^e$, as shown in the following figure:
\[   \input{5-prism.pstex_t}\centering
   \]
By Lemmas~\ref{basic small-2} and \ref{basic small-3}, each colored
3-sided prism used above corresponds to a trivial $\R P^2$-bundle
over $S^1$, and by Lemma~\ref{4-prism1} the three  colored 4-sided
prisms used above correspond to the three nontrivial $T$-bundles
over $S^1$ with monodromy matrices
     $\begin{pmatrix}
          -1 & 0\\ 0 & 1
         \end{pmatrix}$,
         $\begin{pmatrix}
          1 & 0\\ 0 & -1
         \end{pmatrix}$,
        $\begin{pmatrix}
          -1 & 0\\ 0 & -1
         \end{pmatrix}$, respectively. So each $M(P^3(5),\eta_j)$ is equivariantly homeomorphic to
         a non-trivial $N$-bundle
         over $S^1$ with the desired monodromy map.
\end{proof}

Now let us look at how the operation $\widetilde{\sharp^\copyright}$
on $\mathcal{M}$ works. To give a statement in detail, we divide our
discussion into the following three cases.

\begin{enumerate}
\item If we exactly cut out a 2-independent triangular facet from
$(P^3, \lambda)$, then we also need to cut out such a facet from a
colored 3-sided prism $(P^3(3), \tau)$.  According to the colorings
on $P^3(3)$, the topological type of  $M(P^3(3), \tau)$ must be
$S^1\times\mathbb{R}P^2$, so we can cut out a
$(\mathbb{Z}_2)^3$-invariant $\mathbb{R}P^2\times(-1,1)$ from
$S^1\times\mathbb{R}P^2$ with a certain action $\phi$. Then we glue
$M(P^3,\lambda)\backslash(\mathbb{R}P^2\times(-1,1))$ and
$(S^1\times\mathbb{R}P^2,\phi)\backslash(\mathbb{R}P^2\times(-1,1))$
along their boundaries, i.e.,
\begin{align*}
&M(P^3,\lambda)\widetilde{\sharp^\copyright}M(P^3(3), \tau)
=M(P^3,\lambda)\widetilde{\sharp^\copyright}(S^1\times\mathbb{R}P^2,\phi)\\
=&
\big(M(P^3,\lambda)\backslash(\mathbb{R}P^2\times(-1,1))\big)\cup_{\mathbb{R}P^2\sqcup\mathbb{R}P^2}
\big((S^1\times\mathbb{R}P^2,\phi)\backslash(\mathbb{R}P^2\times(-1,1))\big).\end{align*}

\item If we exactly cut out a 2-independent square facet $F$ from $(P^3, \lambda)$, then
we need a colored 4-sided prism $(P^3(4), \tau)$ to do  a coloring
change of $F$. In this case,
 the section in $(P^3, \lambda)$ or $(P^3(4), \tau)$ is a 2-independent square section $S_\square$.
 If $S_\square$ is 2-colorable (i.e,
$S_\square$ corresponds to a disjoint union $T\sqcup T$ by
Lemma~\ref{section}), then by Lemma~\ref{4-prism1} $M(P^3(4), \tau)$
is equivariantly homeomorphic to a twisted $T$-bundle over $S^1$,
and  we can  cut out a $(\mathbb{Z}_2)^3$-invariant $T\times(-1,1)$
from $M(P^3(4), \tau)$. If $S_\square$ is 3-colorable (i.e,
$S_\square$ corresponds to a disjoint union $K\sqcup K$ by
Lemma~\ref{section}), then by Lemma~\ref{4-prism1}, $M(P^3(4),
\tau)$ is equivariantly homeomorphic to a twisted $K$-bundle over
$S^1$, and we can cut out a $(\mathbb{Z}_2)^3$-invariant
$K\times(-1,1)$ from $M(P^3(4), \tau)$. Combining these arguments,
we conclude that if the topological type of $M(P^3(4), \tau)$ is a
twisted $T$-bundle over $S^1$, then
\begin{align*}
&M(P^3,\lambda)\widetilde{\sharp^\copyright}M(P^3(4), \tau)\\
=&\big(M(P^3,\lambda)\backslash(T\times(-1,1))\big)\cup_{T\sqcup
T}\big(M(P^3(4), \tau)\backslash(T\times(-1,1))\big)\end{align*} and
if the topological type of $M(P^3(4), \tau)$ is a twisted $K$-bundle
over $S^1$, then
\begin{align*}
&M(P^3,\lambda)\widetilde{\sharp^\copyright}M(P^3(4), \tau)\\
=& \big(M(P^3,\lambda)\backslash(K\times(-1,1))\big)\cup_{K\sqcup
K}\big(M(P^3(4), \tau)\backslash(K\times(-1,1))\big).
\end{align*}

\item If we exactly cut out a 2-independent pentagonal facet $F$ from
$(P^3, \lambda)$,  then we need a colored 5-sided prism $(P^3(5),
\tau)$ to change the coloring of $F$. Since the section of $(P^3,
\lambda)$ or $(P^3(5), \tau)$ is a 2-independent pentagonal section
$S_\maltese$, by Lemmas~\ref{section} and \ref{5-prism}, $M(P^3(5),
\tau)$ is equivariantly homeomorphic to a twisted $N$-bundle over
$S^1$, and  we can   cut out a $(\mathbb{Z}_2)^3$-invariant
$N\times(-1,1)$ from $M(P^3(5), \tau)$.
 Then the operation
$\widetilde{\sharp^\copyright}$ of $M(P^3,\lambda)$ and $M(P^3(5),
\tau)$ is as follows:
\begin{align*}
&M(P^3,\lambda)\widetilde{\sharp^\copyright}M(P^3(5), \tau)\\ =&
\big(M(P^3,\lambda)\backslash(N\times(-1,1))\big) \cup_{N\sqcup
N}\big(M(P^3(5), \tau)\backslash(N\times(-1,1))\big).
\end{align*}
\end{enumerate}

\begin{rem}
In doing the operation $\widetilde{\sharp^\copyright}$ on a
$M(P^3,\lambda)$,  we cut out a small facet from $(P^3,\lambda)$ and
a bottom facet from a colored $i$-sided prism $(P^3(i),\tau)$,
$i=3,4,5$, and then  glue them together along their sections. There
are similar procedures for $M(P^3,\lambda)$ and $M(P^3(i),\tau)$.
Namely, we first remove an open $(\mathbb{Z}_2)^3$-invariant
$\Sigma$-handlebody $\Sigma\times(-1,1)$ from $M(P^3,\lambda)$ and
$M(P^3(i),\tau)$ respectively where $\Sigma$ is a $\mathbb{R}P^2$,
or a torus, or a Klein bottle, or a
$\mathbb{R}P^2\#\mathbb{R}P^2\#\mathbb{R}P^2$, and then  glue back
the remaining part (i.e., a $(\mathbb{Z}_2)^3$-invariant
$\Sigma$-handlebody $\Sigma\times[-1,1]$) of $M(P^3(i),\tau)$ to
$M(P^3,\lambda)\backslash\Sigma\times(-1,1)$ along their boundaries.
When $i=3$, $M(P^3(3),\tau)$ is a $\mathbb{R}P^2$-bundle over $S^1$
but it is always  trivial, so we can glue back the remaining part of
$M(P^3(3),\tau)$ to $M(P^3,\lambda)\backslash\Sigma\times(-1,1)$
without any twist. However, when $i=4$ or 5, since $M(P^3(i),\tau)$
is always a non-trivial bundle over $S^1$ by Lemmas~\ref{4-prism1},
\ref{4-prism2} and \ref{5-prism},  this means that gluing back
$\Sigma\times[-1,1]$ actually leads to the appearance  of some twist
of $\Sigma\times[-1,1]$, as shown in the following figure:
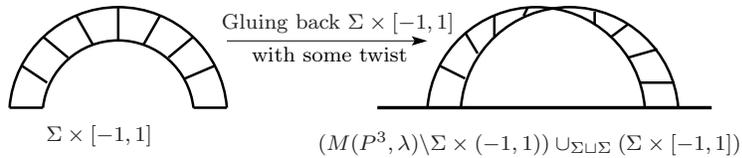
\[   \input{f12.pstex_t}\centering
   \]
\end{rem}

\begin{rem}
 When we do the operations $\widetilde{\natural}$
and $\widetilde{\sharp^\copyright}$ on $\mathcal{M}$, we see that
 after removing  an open $(\mathbb{Z}_2)^3$-invariant desired
3-manifold from $M(\text{\large $\oslash$}, \tau)$ or
$M(P^3(i),\tau)(i=3,4,5)$, the remaining part is still a same type
of $(\mathbb{Z}_2)^3$-invariant 3-manifold with boundary but  admits
a different $(\mathbb{Z}_2)^3$-action. Of course, the actions on
these two 3-manifolds are compatible with the action on
$M(\text{\large $\oslash$}, \tau)$ or $M(P^3(i),\tau)(i=3,4,5)$.
This means that $M(\text{\large $\oslash$}, \tau)$ and
$M(P^3(i),\tau)(i=3,4,5)$ admit  equivariant Heegaard splittings (cf
\cite{H}).
\end{rem}

\section{Application to equivariant cobordism} \label{sec6}

Stong showed in \cite{Stong70} that the $({\Bbb Z}_2)^n$-equivariant
unoriented cobordism class of each closed $(\Z_2)^n$-manifold is
determined by that of its fixed data. This gives the following
result in the special case.

\begin{prop}[Stong]\label{bound}
Suppose that  a closed manifold $M^n$ admits a $(\Z_2)^n$-action
such that its fixed point set is finite.  Then $M^n$ bounds
equivariantly if and only if the tangent representations at fixed
points appear in pairs up to isomorphism.
\end{prop}

Each $n$-dimensional small cover $\pi: M^n\longrightarrow P^n$ has a
finite fixed point set, which just corresponds to the vertex set of
$P^n$. By GKM theory \cite{GKM98}, its tangent representations at
fixed points exactly correspond to a $\text{\rm
Hom}((\Z_2)^n,\Z_2)$-coloring on the 1-skeleton of $P^n$. It is not
difficult to check that this $\text{\rm
Hom}((\Z_2)^n,\Z_2)$-coloring on the 1-skeleton of $P^n$ is
algebraically dual to the $(\Z_2)^n$-coloring on $P^n$, as seen in
the proof of Lemma~\ref{basic small-3}. Therefore, we have that the
$(\Z_2)^n$-colorings of two vertices $v_1, v_2$ in $P^n$ are the
same if and only if the corresponding tangent representations at the
two fixed points $\pi^{-1}(v_1), \pi^{-1}(v_2)$ are isomorphic.
Moreover, by Proposition~\ref{bound} we conclude that

\begin{cor}
Let $\pi: M^n\longrightarrow P^n$ be a small cover over $P^n$. Then
 the $(\Z_2)^n$-colorings of all vertices in $P^n$ appear
in pairs if and only if  $M^n$ bounds equivariantly.
\end{cor}

Now let us look at how six operations work in
$\widehat{\mathcal{M}}$.  Given two classes $[M(P_1^3, \lambda_1)]$
and $[M(P_2^3, \lambda_2)]$  in $\widehat{\mathcal{M}}$, when we do
the operation $\widetilde{\sharp^v}$ on $M(P_1^3, \lambda_1)$ and
$M(P_2^3, \lambda_2)$, we need to cut out two vertices with same
coloring from $(P_1^3, \lambda_1)$ and $(P_2^3, \lambda_2)$
respectively. This means that we exactly cancel two fixed points
with same tangent representation in $M(P_1^3, \lambda_1)\sqcup
M(P_2^3, \lambda_2)$, but this does not change $M(P_1^3,
\lambda_1)\sqcup M(P_2^3, \lambda_2)$ up to equivariant cobordism by
Proposition~\ref{bound}. Thus we have
\begin{lem}\label{cobordism1}
Let $[M(P_1^3, \lambda_1)]$ and $[M(P_2^3, \lambda_2)]$ be two
classes in $\widehat{\mathcal{M}}$. Then $$ [M(P_1^3,
\lambda_1)\widetilde{\sharp^v} M(P_2^3, \lambda_2)]=[M(P_1^3,
\lambda_1)]+[M(P_2^3, \lambda_2)].$$
\end{lem}

By a similar argument, we have

\begin{lem}\label{cobordism2}
Let $[M(P^3, \lambda)]$ be a class in $\widehat{\mathcal{M}}$. Then
\begin{align*}
&[M(P^3,\lambda)\widetilde{\sharp^e}M(P^3(3),
\tau)]=[M(P^3,\lambda)]+[M(P^3(3), \tau)]\\
&[M(P^3,\lambda)\widetilde{\sharp^{eve}}M(P^3_-(3),
\tau)]=[M(P^3,\lambda)]+[M(P^3_-(3), \tau)]\\
&[M(P^3,\lambda)\widetilde{\natural}M(\text{\large $\oslash$},
\tau)]=[M(P^3,\lambda)]\\
&[M(P^3,\lambda)\widetilde{\sharp^\copyright}M(P^3(i),
\tau)]=[M(P^3,\lambda)]+[M(P^3(i), \tau)], i=3,4,5.
\end{align*}
\end{lem}

\begin{rem}
Lemmas~\ref{cobordism1} and \ref{cobordism2} tell us that five
operations $\widetilde{\sharp^v}, \widetilde{\sharp^e},
\widetilde{\sharp^{eve}}$, $\widetilde{\natural}$,
 $\widetilde{\sharp^\copyright}$ have a nice compatibility with the disjoint
 union in the sense of equivariant cobordism.
Notice that clearly $M(\text{\large $\oslash$}, \tau)$ bounds
equivariantly by Proposition~\ref{bound}, so $[M(\text{\large
$\oslash$}, \tau)]=0$ in $\widehat{\mathcal{M}}$.
\end{rem}

However, the operation $\widetilde{\sharp^\triangle}$ is a little
different from other five operations in $\widehat{\mathcal{M}}$. Let
$[M(P_1^3, \lambda_1)]$ and $[M(P_2^3, \lambda_2)]$ be two classes
in $\widehat{\mathcal{M}}$. When we do the operation
$\widetilde{\sharp^\triangle}$ on $M(P_1^3, \lambda_1)$ and
$M(P_2^3, \lambda_2)$, it is possible that we just cut out two
triangular facets  with different colorings from $(P_1^3,
\lambda_1)$ and $(P_2^3, \lambda_2)$ respectively. If this happens,
then we glue the two parts cut out from $(P_1^3, \lambda_1)$ and
$(P_2^3, \lambda_2)$ along their sections, so that we can form a
3-sided prism $P^3(3)$ with a natural induced coloring (denoted by
$\lambda_1\sharp^\triangle\lambda_2$) such that top and bottom
facets are
    colored differently.  Furthermore, this colored 3-sided prism can
    be recovered into a small cover. Thus, by Proposition~\ref{bound}
    we have
    \begin{lem}\label{cobordism3}
Let $[M(P_1^3, \lambda_1)]$ and $[M(P_2^3, \lambda_2)]$ be two
classes in $\widehat{\mathcal{M}}$. Then
\begin{align*}
&[M(P_1^3, \lambda_1)]\widetilde{\sharp^\triangle}[M(P_2^3,
\lambda_2)]\\=&\begin{cases} [M(P_1^3, \lambda_1)]+[M(P_2^3,
\lambda_2)] &\text{if we cut out two triangular}\\
&  \text{facets with same
coloring}\\
[M(P_1^3, \lambda_1)]+[M(P_2^3,
\lambda_2)]+[M(P^3(3),\lambda_1\sharp^\triangle\lambda_2)]&\text{if we cut out two triangular}\\
&  \text{facets with different colorings.}
\end{cases}
\end{align*}
  \end{lem}

Finally, Theorem~\ref{T3} follows immediately from Theorem~\ref{T2}
and Lemmas~\ref{cobordism1}, \ref{cobordism2} and \ref{cobordism3}.

\end{document}